\documentclass{article}

    \usepackage[nonatbib,preprint]{neurips_2020}

\usepackage[utf8]{inputenc} 
\usepackage[T1]{fontenc}    
\usepackage{hyperref}       
\usepackage{url}            
\usepackage{booktabs}       
\usepackage{amsfonts}       
\usepackage{nicefrac}       
\usepackage{microtype}      

\usepackage{stmaryrd}
\usepackage{graphicx}
\usepackage{amsfonts}
\usepackage{amsmath,amsthm}
\usepackage{amssymb}
\usepackage{mathrsfs}
\usepackage{indentfirst}
\usepackage{titlesec}
\usepackage{caption2}
\usepackage{color}
\usepackage[normalem]{ulem}
\usepackage{algorithm}
\usepackage{algorithmic}
\usepackage{lineno}

\usepackage[english]{babel}
\usepackage{cite}
\usepackage{bbm}

\newtheorem{theorem}{Theorem}
\newtheorem{definition}{Definition}

\newtheorem{assumption}{Assumption}
\newtheorem{proposition}{Proposition}[section]
\newtheorem{remark}{Remark}[section]
\newtheorem{lemma}{Lemma}[section]
\newtheorem{corollary}{Corollary}[section]

\numberwithin{equation}{section}

\usepackage{macros-Wenqing}

\title{Stochastic Recursive Momentum Method for Non-Convex Compositional Optimization}

\author{Huizhuo~Yuan
\\
\texttt{huizhuo.yuan@gmail.com}
\AND
Wenqing~Hu
  \\
  Department of Mathematics and Statistics\\
  Missouri University of Science and Technology (formerly University of Missouri, Rolla)\\
  Rolla, MO, 65401, USA \\
  \texttt{huwen@mst.edu} \\
}

\begin{document}

\maketitle

\begin{abstract}
We propose a novel optimization algorithm called STOchastic Recursive Momentum for Compositional (STORM-Compositional) optimization that minimizes the composition of expectations of two stochastic functions, the latter being an optimization problem arising in various important machine learning applications. By introducing the momentum term in the compositional gradient updates, STORM-Compositional operates the stochastic recursive variance-reduced compositional gradients in an exponential-moving average way. This leads to an $\mathcal{O}(\ve^{-3})$ complexity upper bound for STORM-Compositional, that matches the best known complexity bounds in previously announced compositional optimization algorithms. At the same time, STORM-Compositional is a single loop algorithm that avoids the typical alternative tuning between large and small batch sizes, as well as recording of checkpoint gradients, that persist in variance-reduced stochastic gradient methods. This allows considerably simpler parameter tuning in numerical experiments, which
demonstrates the superiority of STORM-Compositional over other stochastic compositional optimization algorithms.
\end{abstract}

\section{Introduction}\label{Sec:Intro}

We revisit here the compositional optimization problem that takes the following general form

\begin{equation}\label{Eq:CompositionalOptimization}
\min\limits_{x\in \mathbb{R}^d}\left\{\Phi(x)=f(g(x))\right\} \ ,
\end{equation}
where 
\begin{equation}\label{Eq:ObjectiveFunction:f-g}
f(y)=\dfrac{1}{n}\sum\limits_{i=1}^n f_i(y) \ ,
 \ g(x)=\dfrac{1}{m}\sum\limits_{j=1}^m g_j(x) \ .
\end{equation}

Here the outer and inner functions $f_i: \mathbb{R}^l\rightarrow \mathbb{R}$ and $g_j: \mathbb{R}^d\rightarrow \mathbb{R}^l$ are smooth but not 
necessarily convex. Such compositional optimization is important as it can be formulated to fit many practical machine learning problems, examples include risk-adverse portfolio management (see \cite{Portfolio-Optimization}, \cite{SARAH-SCGD}), reinforcement learning (see \cite{[Sutton-Barto]}) and stochastic neighborhood embedding (see \cite{HintonSNE}, \cite{SCVR}), etc. . To solve \eqref{Eq:CompositionalOptimization} using a gradient-based algorithm, we form the \textit{compositional gradient} of $\Phi$ given by 

\begin{equation}\label{Eq:CompositionalGradient}
\nabla \Phi(x)=(\partial g(x))^T\nabla f(g(x)) \ .
\end{equation}

Here $\partial g: \mathbb{R}^d\rightarrow \mathbb{R}^l$ is the Jacobian matrix of the function $g$, and $\nabla f\in \mathbb{R}^l$ is the gradient of the function $f$. 
Generally, if $q=q(x)$ is some quantity and there is an existing scheme $ \ \widehat{} : q\rightarrow \widehat{q}$ that turns the quantity $q$ into its stochastic estimator $\widehat{q}$, then the compositional gradient estimator can be built simply by first obtaining $\widehat{g}$, $\widehat{\partial g}$, $\widehat{\nabla f}$ and then compose $\widehat{\nabla\Phi}=(\widehat{\partial g})^T\widehat{\nabla f}(\widehat{g})$. This idea has been adopted in many previously proposed algorithms, such as SCGD and Acc-SCGD (see \cite{SCGD-Wang2017}), ASC-PG (see \cite{Acc-SCGD-Wang2017}), where the stochastic gradient estimates are obtained via vanilla SGD or accelerated SGD. Later, variance-reduced stochastic gradient estimators are incorporated into this scheme, such as SCVR (see \cite{SCVR}), VRSC-PG (see \cite{VRSC-PG}) and SARAH-Compositional (abbreviated as SARAH-C throughout, see \cite{SARAH-SCGD}). The so-far best upper bound for the IFO (Incremental First order Oracle) complexity of compositional optimization algorithms that can reach $\ve$-accuracy is given by $\mathcal{O}(\ve^{-3})$ via SARAH-C (see \cite{SARAH-SCGD}), which uses successive SARAH (see \cite{SARAH}) estimators (a kind of variance reduced stochastic estimator) of the variables $g$, $\partial g$, $\nabla f$ and finally composed them using the aforementioned scheme to obtain an estimator for $\nabla \Phi$. 

As a variance-reduced stochastic gradient method, SARAH-C shares the common feature that it uses typical alternative tuning between large and small batch sizes, as well as recording of checkpoint gradients, which persist in variance-reduced stochastic gradient methods. This causes difficulties in parameter tuning, since balancing various unknown problem parameters exactly in order to obtain improved performance is a challenging and delicate task. To resolve this challenge, a recent work STOchastic Recursive Momentum estimator (STORM, \cite{STORM}) proposes to tune the momentum term within a single-loop algorithm to reach the same efficiency for variance-reduction. Based on this innovation, we propose here yet another compositional optimization algorithm called STOchastic Recursive Momentum for Compositional (STORM-Compositional) optimization. By introducing the momentum term in the compositional gradient updates, STORM-Compositional operates the stochastic recursive compositional gradients in an exponential-moving average way. This leads to the same $\mathcal{O}(\ve^{-3})$ complexity upper bound, that matches the best known complexity bounds in SARAH-C. At the same time, STORM-Compositional is a single loop algorithm that avoids batchsize/learning rate tuning and checkpoint recording, as compared with SARAH-C. This allows considerably simpler parameter tuning in numerical experiments, which
demonstrates the superiority of STORM-Compositional over other stochastic compositional optimization algorithms.

The paper is organized as follows: Section \ref{Sec:Setup-Algorithm} introduces the mathematical assumptions, set-up of the problem as well as the STORM-Compositional Algorithm. Section \ref{Sec:ConvergenceAnalysis} analyzes the convergence
of STORM-Compositional and shows that under appropriate parameter setting it can reach $\mathcal{O}(\ve^{-3})$ complexity. Section \ref{Sec:Experiment} demonstrates a numerical experiment of STORM-Compositional over SARAH-C and other known compositional optimization algorithms based on the example of portfolio management problem, that validates the effectiveness of STORM-Compositional. 

\textbf{Notations.} Throughout the paper, we denote by $\|x\|$ the Euclidean norm of a vector $x\in \R^d$. If $M\in \R^{d_1\times d_2}$ is a matrix of size $d_1$ by $d_2$, then $\|M\|$ denotes the operator norm of $M$ induced by the Euclidean norm. We also use $\|M\|_F$ to denote the Frobenious norm of $M$. The standard inner product in Euclidean space $\R^d$ is denoted by $\langle x, y \rangle$. Probabilities and expectations are denoted by $\Prob$ and $\E$. We denote $p_n = \cO(q_n)$ (or $p_n \sim \cO(q_n)$) if there exist some constants $0 < c < C < \infty$ such that $cq_n \leq p_n \leq Cq_n$ as $n$ becomes large. If only one-sided inequality holds, say $p_n\leq Cq_n$, then we denote $p_n\lesssim \cO(q_n)$. We denote by $q(x, \cB)$ to be the minibatch stochastic estimator of the object $q(x)=\dfrac{1}{n}\sum\li_{i=1}^n q_i(x)$ under minibatch $\cB=\{i_1, ..., i_B\}\subseteq \{1,2,...,n\}$ with batchsize $B$, i.e. $q(x, \cB)=\dfrac{1}{B}\sum\li_{i\in \cB} q_i(x)$. Other notations are explained at their first appearances. 

\textbf{Statement of Contributions.} Huizhuo Yuan brought Wenqing Hu's attention to this problem and participated in one small discussion when Wenqing Hu raised the question on the mini-batch sampling with replacement, and another discussion about a question raised by Wenqing Hu regarding the missing bound for $\|\boldG_t\|$ in \cite{arXivSARAH-SCGD}, the latter leading Wenqing Hu to look at \cite{XiaoEtAlCompositional} instead of \cite{arXivSARAH-SCGD}. Wenqing Hu performed all the mathematical proofs in this work and worked on the numerical experiment for SNE. Wenqing Hu worked on the overall structure of the presentation and the write-up of the whole paper. 

\textbf{Acknowledgement.}  The first two numerical experiments of this work are done by Dr. Jiaojiao Yang from the School of Mathematics and Statistics, Anhui Normal University. Due to no initial involvement into the project and upon graceful agreement with Jiaojiao Yang, she is not listed as an author here. Still, Wenqing Hu would like to thank Jiaojiao Yang for the hard work in the numerical experiments.

\section{Assumptions and the STORM-Compositional Algorithm}\label{Sec:Setup-Algorithm}

To compare, our assumptions about $g$, $\partial g$ , $\nabla f$ and the objective function $\Phi$ are following the SARAH-C paper \cite{SARAH-SCGD}, \cite{arXivSARAH-SCGD}.

\begin{assumption}[Finite Gap]\label{Assumption:FiniteGap}
We assume that the algorithm is initialized at $x_0\in \mathbb{R}^d$ with
\begin{equation}\label{Assumption:FiniteGap:Eq:FiniteGap}
\Delta:=\Phi(x_0)-\Phi^*<\infty \ ,
\end{equation}
where $\Phi^*$ denotes the global minimum value of $\Phi(x)$.
\end{assumption}

\begin{assumption}[Smoothness]\label{Assumption:Smoothness}
There exists Lipschitz constants $L_f, L_g, L_\Phi$ such that for $i\in \{1,...,n\}$, $j\in \{1,...,m\}$ we have 
\begin{equation}\label{Assumption:Smoothness:Eq:Lipschitz}
\begin{array}{lll}
\|\partial g_j(x)-\partial g_j(x')\|_F&\leq L_g\|x-x'\| & \text{ for } x, x'\in \mathbb{R}^d  \ ,
\\
\|\nabla f_i(y)-\nabla f_i(y')\| & \leq L_f\|y-y'\|& \text{ for } y, y'\in \mathbb{R}^l \ ,
\\
\|(\partial g_j(x))^T\nabla f_i(g(x))-(\partial g_j(x'))^T\nabla f_i(g(x'))\| & \leq L_\Phi\|x-x'\|& \text{ for } x, x'\in \mathbb{R}^d \ .
\end{array}
\end{equation}
\end{assumption}

\begin{assumption}[Boundedness]\label{Assumption:Boundedness}
There exist boundedness constants $M_g , M_f>0$ such that for $i\in \{1,...,n\}$ and $j\in \{1,...,m\}$ we have
\begin{equation}\label{Assumption:Boundedness:Eq:Bounds}
\begin{array}{ll}
\|\partial g_j(x)\|\leq M_g & \text{ for } x\in \mathbb{R}^d  \ ,
\\
\|\nabla f_i(y)\|\leq M_f & \text{ for } y\in \mathbb{R}^l \ .
\end{array}
\end{equation}
\end{assumption}
\vspace{-2mm}
Notice that \eqref{Assumption:Boundedness:Eq:Bounds} directly implies that for any $j\in \{1,2,...,m\}$ we have
\vspace{-1mm}
\begin{equation}\label{Eq:Lipschitz-g}
\|g_j(x)-g_j(x')\|\leq M_g\|x-x'\| \ \text{ for } x, x'\in \mathbb{R}^d \ .
\end{equation}
\vspace{-0.2mm}
Also notice that under the above two assumptions, a choice of $L_\Phi$ can be expressed as a polynomial of $L_f, L_g, M_f, M_g$. For clarity purposes in the rest of this
paper, we adopt the following typical choice of $L_\Phi$, that
\begin{equation}\label{Eq:Choice-L-Phi}
L_\Phi=M_fL_g+M_g^2L_f \ .
\end{equation}

\begin{assumption}[Bounded Variance]\label{Assumption:FiniteVariance}
We assume that there exist positive constants $H_1$, $H_2$ and $H_3$ as the upper bounds on the variance of the functions $\nabla f(y)$, $\partial g(x)$ and $g(x)$, respectively, such that
\begin{equation}\label{Assumption:FiniteVariance:Eq:BoundsVariance}
\begin{array}{lll}
\mathbf{E}\|\nabla f_i(y)-\nabla f(y)\|^2 & \leq H_1 & \text{ for } y\in \mathbb{R}^l \ ;
\\
\mathbf{E}\|\partial g_j(x)-\partial g(x)\|_F^2 & \leq H_2 & \text{ for } x\in \mathbb{R}^d \ ;
\\
\mathbf{E}\|g_j(x)-g(x)\|^2 & \leq H_3 & \text{ for } x\in \mathbb{R}^d \ . 
\end{array}
\end{equation}
\end{assumption}
\vspace{-1mm}
Here $i, j$ are sampled uniformly randomly from $\{1,...,n\}$ and $\{1,...,m\}$. Notice that here in \eqref{Assumption:Smoothness:Eq:Lipschitz} and \eqref{Assumption:FiniteVariance:Eq:BoundsVariance} we slightly strengthen the
assumptions regarding $\partial g$ by adopting the Frobenius norm. 

The key idea in momentum-based methods 
is to replace the gradient estimator in standard gradient-based optimization by an exponential moving average of the gradient estimators of all previous iteration steps. For example, we can design the following iteration scheme:
\vspace{-1mm}
\begin{equation}\label{Eq:MomentumMethodsGeneral}
\begin{array}{ll}
\boldsymbol{d}_{t}&= (1-a)\boldsymbol{d}_{t-1}+a\nabla f(x_t, \mathcal{B}_{t}) \ ,
\\
x_t & =x_{t-1}-\eta \boldsymbol{d}_t \ ,
\end{array}
\end{equation}
Here $
\nabla f(x, \mathcal{B})=B^{-1}\sum_{i\in \mathcal{B}_t}\nabla f_i(x)
$
is the a minibatch stochastic gradient estimator for the objective gradient $\nabla f(x)=n^{-1}\sum_{i=1}^n \nabla f_i(x)$, with minibatch $\mathcal{B}_t$ and batchsize $B$, i.e., $\mathcal{B}_t$ is a size $B$ random sample drawn from the index set $\{1,...,n\}$ independently in $t$. The STORM estimator (see \cite{STORM}) incorporates ideas in SARAH (see \cite{SARAH}) estimator to modify the momentum updates as

\begin{equation}\label{Eq:STORM-update}
\begin{array}{ll}
\boldsymbol{d}_{t}&= (1-a)\boldsymbol{d}_{t-1}+a\nabla f(x_t, \mathcal{B}_{t})+(1-a)(\nabla f(x_t, \mathcal{B}_t)-\nabla f(x_{t-1}, \mathcal{B}_t)) \ ,
\\
x_t & =x_{t-1}-\eta \boldsymbol{d}_t \ .
\end{array}
\end{equation}

We can think of the momentum iteration in \eqref{Eq:STORM-update} as an ``exponential moving average" version of SARAH, that is
$\boldsymbol{d}_{t}= (1-a)\left[\boldsymbol{d}_{t-1}+(\nabla f(x_t, \mathcal{B}_t)-\nabla f(x_{t-1}, \mathcal{B}_t))\right]+a\nabla f(x_t, \mathcal{B}_{t}) \ .
$
Here the first part is like SARAH update, but it is given a weight distribution  with the stochastic gradient update. Moreover, different from the SARAH update, the STORM update does not have to compute checkpoint gradients. Rather, variance-reduction is achieved by tuning the weight parameter $a$ and the learning rate $\eta$ appropriately. Taking into account the general scheme for obtaining compositional gradients, we can design the STORM-Compositional Algorithm, see Algorithm \ref{Alg:STORM-Compositional}.

\begin{algorithm}
	\caption{STORM-Compositional: STOchastic Recursive Momentum for Compositional Gradients}
	\label{Alg:STORM-Compositional}
	\begin{algorithmic}[1]
		\STATE \textbf{Input}: Initial point $x_0$; learning rate $\eta>0$, parameters $a_g, a_{\partial g}, a_{\Phi}\in (0,1)$, batchsizes $B_g, B_{\partial g}, B_f, S_g, S_{\partial g}, S_f$, desired precision $\varepsilon>0$
		\STATE Sample "with replacement" minibatches $\mathcal{S}_0^g, \mathcal{S}_0^{\partial g}, \mathcal{S}_0^f$ under given batch sizes $S_g$, $S_{\partial g}$ and $S_{f}$
		\STATE Pick $\boldsymbol{g}_0\leftarrow g(x_0, \mathcal{S}_0^g)$
		, $\boldsymbol{G}_0\leftarrow \partial g(x_0, \mathcal{S}_0^{\partial g})$ ,
		$\boldsymbol{F}_0\leftarrow (\boldsymbol{G}_0)^T\nabla f(\boldsymbol{g}_0, \mathcal{S}_0^f)$
		\FOR {$t=0,1,2,...,T$} 
			\STATE $\widetilde{x}_{t+1}\leftarrow x_t-\eta\boldsymbol{F}_t$
			\STATE $x_{t+1}=x_t+\gamma_t(\widetilde{x}_{t+1}-x_t)$ where $\gamma_t=\min\left\{\dfrac{\eta \varepsilon}{\|\widetilde{x}_{t+1}-x_t\|}, \dfrac{1}{2}\right\}$
		    \STATE Sample "with replacement" minibatches $\mathcal{B}_{t+1}^g, \mathcal{B}_{t+1}^{\partial g}, \mathcal{B}_{t+1}^f$ under given batch sizes $B_g$ , $B_{\partial g}$ and $B_f$
			\STATE Calculate $$\begin{array}{ll}\boldsymbol{g}_{t+1}&\leftarrow (1-a_g)\boldsymbol{g}_t+a_g g(x_{t+1}, \mathcal{B}_{t+1}^g)+(1-a_g)(g(x_{t+1}, \mathcal{B}_{t+1}^g)-g(x_t, \mathcal{B}_{t+1}^g))\end{array}$$
			\STATE Calculate $$\begin{array}{ll}\boldsymbol{G}_{t+1} &\leftarrow(1-a_{\partial g})\boldsymbol{G}_t+a_{\partial g} \partial g (x_{t+1}, \mathcal{B}_{t+1}^{\partial g})+(1-a_{\partial g})(\partial g(x_{t+1}, \mathcal{B}_{t+1}^{\partial g})-\partial g(x_t, \mathcal{B}_{t+1}^{\partial g}))\end{array}$$
			\STATE Calculate $$\begin{array}{ll}\boldsymbol{F}_{t+1}& \leftarrow(1-a_{\Phi})\boldsymbol{F}_t+a_{\Phi} (\boldsymbol{G}_{t+1})^T\nabla f(\boldsymbol{g}_{t+1}, \mathcal{B}^f_{t+1})\\ & \qquad +(1-a_{\Phi})\left[(\boldsymbol{G}_{t+1})^T\nabla f(\boldsymbol{g}_{t+1}, \mathcal{B}^f_{t+1})-(\boldsymbol{G}_{t})^T\nabla f(\boldsymbol{g}_{t}, \mathcal{B}^f_{t+1})\right]\end{array}$$
		\ENDFOR
		\STATE \textbf{Output}: $\widehat{x}$ sampled uniformly randomly from $x_1,..., x_T$ (In practice, set $\widehat{x}=x_T$)
	\end{algorithmic}
\end{algorithm}

The normalization Step 6 in Algorithm \ref{Alg:STORM-Compositional} is borrowed from \cite{SPIDER} and \cite{XiaoEtAlCompositional}, and it enables the control of the growth of $\boldG_t$ norm in Step 9, that will be crucial for the control of estimation error in general \footnote{Notice that the bound for $\|\boldG_t\|$ has been in fact missed in the algorithm and convergence analysis of \cite{arXivSARAH-SCGD}.}. We will see later that the estimation error $\mathbf{E}\|\nabla \Phi(x_{t})-\boldsymbol{F}_t\|^2$ of Algorithm \ref{Alg:STORM-Compositional} can be controlled by a proper choice
of the parameters $a_g, a_{\partial g}, a_{\Phi}$, the batchsizes $B_g, B_{\partial g}, B_f, S_g, S_{\partial g}, S_f$ and the learning rate $\eta$. The desired precision parameter $\varepsilon$ is used to cutoff the accuracy of the solution $\widehat{x}$.

It is also important to notice that the minibatches sampled in Algorithm \ref{Alg:STORM-Compositional} are "with replacement" minibatches, i.e., the minibach is formed by randomly sampling indexes from the pool $\{1,2,...,n\}$ or $\{1,2,...,m\}$ with replacement. This is mainly used to control the variances caused by mini-batch sampling, see the proof of Theorem \ref{Proposition:IFO-3eps-accuracy}.


\section{Convergence Analysis of the STORM-Compositional Algorithm}\label{Sec:ConvergenceAnalysis} 

The complexity of Algorithm \ref{Alg:STORM-Compositional} is defined using the $\ve$ (or $\mathcal{O}(\ve)$)-accurate solutions and the IFO (Incremental First-order Oracle)
framework (see \cite{IFO}) defined below.

\begin{definition}[$\varepsilon$ and $\mathcal{O}(\varepsilon)$-accurate solutions]\label{Definition:eps-accuracy}
We call a solution $\widehat{x}$ an $\varepsilon$-accurate solution if we have
\begin{equation}\label{Definition:eps-accuracy:Eq:eps-accuracy}
\mathbf{E}\|\nabla \Phi(\widehat{x})\|\leq \varepsilon \ .    
\end{equation}
We call a solution 
$\widehat{x}$ an order-$\varepsilon$-accurate ($\mathcal{O}(\varepsilon)$-accurate) solution if we have
\begin{equation}\label{Definition:eps-accuracy:Eq:Order-eps-accuracy}
\mathbf{E}\|\nabla \Phi(\widehat{x})\|\lesssim \mathcal{O}(\varepsilon) \ .    
\end{equation}
\end{definition}

\begin{definition}[IFO complexity]\label{Definition:IFOComplexity}
For any function $f, g$ considered in Algorithm \ref{Alg:STORM-Compositional}, an IFO takes an index $i\in \{1,2,...,n\}$ or $j\in \{1,2,...,m\}$
and a point $y\in \mathbb{R}^l$ or a point $x\in \mathbb{R}^d$, and returns the pair $(f_i(y), \nabla f_i(y))$ or $(g_j(x), \partial g_j(x))$.
\end{definition}

Thus the IFO-complexity of Algorithm \ref{Alg:STORM-Compositional} can be calculated in the following simple Lemma.

\begin{lemma}[Total IFO complexity]\label{Lemma:IFO-Complexity-Precise}
The total IFO complexity of Algorithm \ref{Alg:STORM-Compositional} is given by
\begin{equation}\label{Lemma:IFO-Complexity-Precise:Eq:IFO}
\text{IFO}=S_g+S_{\partial g}+S_f+ T(B_g+B_{\partial g}+B_f) \ .
\end{equation}
\end{lemma}

The following Proposition gives an error bound for the estimation error for $\E\|\grad \Phi(\widehat{x})\|$, where $\widehat{x}$ is the output of Algorithm \ref{Alg:STORM-Compositional}. The proof is left to Supplementary Materials.

\begin{proposition}[Bound for the Estimation Error]\label{Theorem:MainTheoremConvergence}
Pick $\eta=\dfrac{1}{L_\Phi}$, if
\vspace{-0.1cm}
\begin{equation}\label{Theorem:MainTheoremConvergence:Eq:ErrorAssumption}
\dfrac{1}{T}\sum\limits_{t=0}^{T-1}\mathbf{E}\|\boldsymbol{F}_t-\nabla \Phi(x_t)\|^2\leq A\varepsilon^2 \ ,
\end{equation}
for some $A>0$, then we have the estimate
\begin{equation}\label{Theorem:MainTheoremConvergence:Eq:FinalEstimate}
\mathbf{E}\|\nabla \Phi(\widehat{x})\|\leq \dfrac{2L_\Phi\Delta}{T\varepsilon}+\left(\dfrac{1}{2}+A+\sqrt{A}\right)\varepsilon \ .
\end{equation}
\end{proposition}

Combining Lemma \ref{Lemma:IFO-Complexity-Precise} and Proposition \ref{Theorem:MainTheoremConvergence}, we can show the following Theorem on reaching $\varepsilon$-accurate solutions for Algorithm \ref{Alg:STORM-Compositional} within IFO complexity $\lesssim \mathcal{O}(\varepsilon^{-3})$.

\begin{theorem}[IFO complexity to reach $\varepsilon$-accuracy]\label{Proposition:IFO-3eps-accuracy}
Under the choice of parameters $a_g=\alpha_g\varepsilon$, $a_{\partial g}=\alpha_{\partial g}\varepsilon$, $a_{\Phi}=\alpha_{\Phi}\varepsilon$, $B_g=\beta_g\varepsilon^{-1}$, $B_{\partial g}=\beta_{\partial g}\varepsilon^{-1}$, $B_f=\beta_f\varepsilon^{-1}$, $S_g=\gamma_g\varepsilon^{-1}$, $S_{\partial g}=\gamma_{\partial g}\varepsilon^{-1}$, $S_f=\gamma_f\varepsilon^{-1}$ and $T=\dfrac{32}{3}L_\Phi \Delta \varepsilon^{-2}$, $\eta=\dfrac{1}{L_\Phi}$, we have $\mathbf{E}\|\nabla \Phi(\widehat{x})\|\leq \varepsilon$ for the output $\widehat{x}$ of Algorithm \ref{Alg:STORM-Compositional} if the following condition holds
\begin{equation}\label{Proposition:IFO-3eps-accuracy:Eq:AccuracyCondition}
\begin{array}{ll}
&\left\{\dfrac{36M_f^2L_g^2}{\alpha_{\Phi}\beta_f}\left[\dfrac{4}{
\alpha_{\partial g}\beta_{\partial g}}+\dfrac{2\varepsilon}{\beta_{\partial g}}+1\right]+\dfrac{36L_f^2M_g^2}{\alpha_{\Phi}\beta_f}\left(2M_g+\dfrac{L_g}{L_\Phi\alpha_{\partial g}}\right)\left[\dfrac{4}{\alpha_g\beta_g}+\dfrac{2\varepsilon}{\beta_g}+1\right]\right.
\\
& \qquad \left.+\dfrac{6M_f^2L_g^2}{\alpha_{\partial g}\beta_{\partial g}}+\dfrac{6M_g^4L_f^2}{\alpha_g\beta_g}\right\}\dfrac{1}{L^2_\Phi}
\\
& +\dfrac{6M_f^2H_2}{\frac{32}{3}L_\Phi\Delta\alpha_{\partial g}\gamma_{\partial g}}\left(\dfrac{24}{\alpha_\Phi \beta_f }+1\right)
+\dfrac{6L_f^2H_3}{\frac{32}{3}L_\Phi\Delta\alpha_g\gamma_g}\left(\dfrac{24}{\alpha_{\Phi}\beta_f}\left(2M_g+\dfrac{L_g}{\
L_\Phi\alpha_{\partial g}}\right)+M_g^2\right)
\\
& + \dfrac{3M_g^2H_1}{\frac{32}{3}L_\Phi\Delta\alpha_{\Phi}\gamma_f}+\dfrac{72M_f^2\alpha_{\partial g}H_2}{\alpha_{\Phi}\beta_f\beta_{\partial g}}(2+\alpha_{\partial g}\varepsilon)+\dfrac{72L_f^2\alpha_gH_3}{\alpha_{\Phi}\beta_f\beta_g}\left(2M_g+\dfrac{L_g}{L_\Phi\alpha_{\partial g}}\right)(2+\alpha_g\varepsilon)
\\
& + \dfrac{6\alpha_\Phi}{\beta_f}\left(2M_g+\dfrac{L_g}{L_\Phi \alpha_{\partial g}}\right)^2H_1+\dfrac{6M_f^2\alpha_{\partial g}H_2}{\beta_{\partial g}}+\dfrac{6M_g^2L_f^2\alpha_gH_3}{\beta_g} 
\\
\leq & \dfrac{1}{16} \ .
\end{array}
\end{equation}
The total IFO complexity of Algorithm \ref{Alg:STORM-Compositional} in this case is given by 
\begin{equation}\label{Proposition:IFO-3eps-accuracy:Eq:Total-IFO-complexity}
\text{IFO}=(\gamma_g+\gamma_{\partial g}+\gamma_f)\varepsilon^{-1}+\dfrac{32}{3}L_\Phi\Delta(\beta_g+\beta_{\partial g}+\beta_f)\varepsilon^{-3} \ .
\end{equation}
\end{theorem}

\begin{proof}
Since $\nabla \Phi(x)=(\partial g(x))^T\nabla f(g(x))$, we can estimate $\sum\limits_{t=0}^{T-1}\mathbf{E}\|\boldsymbol{F}_t-\nabla \Phi(x_t)\|^2$ using the expansion

\begin{equation}\label{Proposition:AsymptoticEpsAccuracy:Eq:EstimateF-To-gradPhi}
\begin{array}{ll}
&\sum\limits_{t=0}^{T-1}\mathbf{E}\|\boldsymbol{F}_t-\nabla \Phi(x_t)\|^2
\\
\leq &3\sum\limits_{t=0}^{T-1}\mathbf{E}\|\boldsymbol{F}_t-(\boldsymbol{G}_t)^T\nabla f(\boldsymbol{g}_t)\|^2+3\sum\limits_{t=0}^{T-1}\mathbf{E}\|(\boldsymbol{G}_t)^T\nabla f(\boldsymbol{g}_t)-(\partial g(x_t))^T\nabla f(\boldsymbol{g}_t)\|^2
\\
& \qquad +3\sum\limits_{t=0}^{T-1}\mathbf{E}\|(\partial g(x_t))^T\nabla f(\boldsymbol{g}_t)-(\partial g(x_t))^T\nabla f(g(x_t))\|^2
\\
\leq &3\sum\limits_{t=0}^{T-1}\mathbf{E}\|\boldsymbol{F}_t-(\boldsymbol{G}_t)^T\nabla f(\boldsymbol{g}_t)\|^2+3M_f^2\sum\limits_{t=0}^{T-1}\mathbf{E}\|\boldsymbol{G}_t-\partial g(x_t)\|_F^2 +3M_g^2L_f^2\sum\limits_{t=0}^{T-1}\mathbf{E}\|\boldsymbol{g}_t-g(x_t)\|^2 \ .
\end{array}
\end{equation}

We will prove in Corollary A.1 and Lemma A.7 in the Supplementary Materials the following bounds
\vspace{-1mm}
\begin{equation}\label{Corollary:ErrorEstimateToBottom:Eq:g}
\sum\limits_{t=0}^{T-1}\mathbf{E}\|\boldsymbol{g}_t-g(x_t)\|^2
\leq\dfrac{2}{a_g}\left[\dfrac{1}{B_g}M_g^2\sum\limits_{t=0}^{T-1}\mathbf{E}\|x_{t+1}-x_t\|^2+\dfrac{Ta_g^2H_3}{B_g}+\mathbf{E}\|\boldsymbol{g}_0-g(x_0)\|^2\right] \ ,
\end{equation}
\vspace{-0.2cm}
\begin{equation}\label{Corollary:ErrorEstimateToBottom:Eq:partialg}
\sum\limits_{t=0}^{T-1}\mathbf{E}\|\boldsymbol{G}_t-\partial g(x_t)\|_F^2
\leq\dfrac{2}{a_{\partial g}}\left[\dfrac{1}{B_{\partial g}}L_g^2\sum\limits_{t=0}^{T-1}\mathbf{E}\|x_{t+1}-x_t\|^2+\dfrac{Ta_{\partial g}^2H_2}{B_{\partial g}}+\mathbf{E}\|\boldsymbol{G}_0-\partial g(x_0)\|_F^2\right] \ ,
\end{equation}
\vspace{-0.3cm}
\begin{equation}\label{Lemma:ErrorEstimateToBottom:gradPhi:Eq:Error}
\begin{array}{ll}
&\sum\limits_{t=0}^{T-1}\mathbf{E}\|\boldsymbol{F}_t-(\boldsymbol{G}_t)^T\nabla f(\boldsymbol{g}_t)\|^2
\\
\leq & \dfrac{4M_f^2}{a_{\Phi}B_f}\left\{\left[\dfrac{12L_g^2}{a_{\partial g}B_{\partial g}}+\dfrac{6L_g^2}{B_{\partial g}}+3L_g^2\right]\sum\limits_{t=0}^{T-1}\mathbf{E}\|x_{t+1}-x_t\|^2 \right.
\\
& \left. \qquad \qquad 
+\dfrac{12Ta_{\partial g}H_2}{B_{\partial g}}+\dfrac{12}{a_{\partial g}}\mathbf{E}\|\boldsymbol{G}_0-\partial g(x_0)\|_F^2 +\dfrac{6T}{B_{\partial g}}(a_{\partial g})^2H_2\right\} \
\\
&+\dfrac{4L_f^2}{a_{\Phi}B_f}\left(2M_g+\dfrac{L_g\eta\varepsilon}{a_{\partial g}}\right)\left\{\left[\dfrac{12M_g^2}{a_gB_g}+\dfrac{6M_g^2}{B_g}+3M_g^2\right]\sum\limits_{t=0}^{T-1}\mathbf{E}\|x_{t+1}-x_t\|^2\right.
\\
& \left. \qquad \qquad \qquad \qquad \qquad \qquad
+\dfrac{12Ta_gH_3}{B_g}+\dfrac{12}{a_g}\mathbf{E}\|\boldsymbol{g}_0-g(x_0)\|^2 +\dfrac{6T}{B_g}(a_g)^2H_3\right\}
\\
&+\dfrac{2T}{B_f}(a_{\Phi})\left(2M_g+\dfrac{L_g\eta\varepsilon}{a_{\partial g}}\right)^2
H_1+\dfrac{1}{a_{\Phi}}\mathbf{E}\|\boldsymbol{F}_0-(\boldsymbol{G}_0)^T\nabla f(\boldsymbol{g}_0)\|^2 \ . 
\end{array}
\end{equation}

From Lemma A.4 in the Supplementary Materials we know that $\|x_{t+1}-x_t\|\leq \eta \ve$. Moreover, we also notice that since the initial batches are sampled with replacement we have $\mathbf{E}\|\boldsymbol{g}_0-g(x_0)\|^2\leq \dfrac{H_3}{S_g}$, $\mathbf{E}\|\boldsymbol{G}_0-\partial g(x_0)\|_F^2\leq \dfrac{H_2}{S_{\partial g}}$ and 
$\mathbf{E}\|\boldsymbol{F}_0-(\boldsymbol{G}_0)\nabla f(\boldsymbol{g}_0)\|^2\leq\dfrac{M_g^2H_1}{S_f}$. Thus combining all these with \eqref{Proposition:AsymptoticEpsAccuracy:Eq:EstimateF-To-gradPhi}, \eqref{Corollary:ErrorEstimateToBottom:Eq:g}, \eqref{Corollary:ErrorEstimateToBottom:Eq:partialg}, \eqref{Lemma:ErrorEstimateToBottom:gradPhi:Eq:Error} as well as Lemma \ref{Lemma:IFO-Complexity-Precise}, Proposition \ref{Theorem:MainTheoremConvergence} (taking $A=\dfrac{1}{16}$), we obtain the statement of the Theorem.
\end{proof}
\vspace{-2mm}
The $a$ and batchsizes parameters $a_g, a_{\pt g}, a_\Phi, B_g, B_{\pt g}, B_f, S_g, S_{\pt g}, S_f$ can be chosen so that condition \eqref{Proposition:IFO-3eps-accuracy:Eq:AccuracyCondition} is satisfied. We leave the proof of the following Proposition to Supplementary Materials.

\begin{proposition}[Choice of parameters to reach $\varepsilon$-accurate solution]\label{Proposition:Choice-Paramaters}
We pick the accuracy parameter 
\begin{equation}\label{Proposition:Choice-Paramaters:Eq:epsilon}
0<\varepsilon<\min\left(1, 72L_{\Phi}M_gL_f^2\sqrt{H_3}, \dfrac{72L_{\Phi}M_f^2\sqrt{H_2}}{L_g}\right) \ ,    
\end{equation}
and we set
\begin{equation}\label{Proposition:Choice-Paramaters:Eq:K0}
K_0=\dfrac{4M_f^2}{B_f}\left(\dfrac{1}{24M_f^2}+\dfrac{3L_g^2}{L_{\Phi}^2}+1\right)+\dfrac{4L_f^2}{B_f}(2M_g+\sqrt{H_2})\left(\dfrac{1}{24M_g^2L_f^2}+\dfrac{3M_g^2}{L_{\Phi}^2}+1\right) \ .
\end{equation}
Then we choose
\begin{equation}\label{Proposition:Choice-Paramaters:Eq:a}
a_g= \dfrac{M_g}{L_\Phi\sqrt{H_3}}\varepsilon,  
a_{\partial g}= \dfrac{L_g}{L_\Phi\sqrt{H_2}}\varepsilon, 
a_{\Phi}= \sqrt{\dfrac{K_0}{2(2M_g+\sqrt{H_2})H_1}}\varepsilon;
\end{equation}
\begin{equation}\label{Proposition:Choice-Paramaters:Eq:B}
B_g= \dfrac{864M_g^3L_f^2\sqrt{H_3}}{L_\Phi}\varepsilon^{-1}, B_{\partial g}=\dfrac{864M_f^2L_g\sqrt{H_2}}{L_\Phi}\varepsilon^{-1}, B_f=432\sqrt{2(2M_g+\sqrt{H_2})H_1K_0} \varepsilon^{-1} \ ;
\end{equation}
\vspace{-0.2cm}
\begin{equation}\label{Proposition:Choice-Paramaters:Eq:S}
S_g=\dfrac{81M_gL_f^2H_3^{3/2}}{\Delta}\varepsilon^{-1}, S_{\partial g}=\dfrac{81M_f^2H_2^{3/2
}}{\Delta L_g}\varepsilon^{-1}, S_f=\dfrac{M_g^2H_1}{\frac{32}{3}L_{\Phi}\Delta K_0}432\sqrt{2(2M_g+\sqrt{H_2})H_1K_0} \varepsilon^{-1} \ .
\end{equation}
We also pick 
\begin{equation}\label{Proposition:Choice-Paramaters:Eq:eta-T}
\eta=\dfrac{1}{L_{\Phi}} \ , \ T=\dfrac{32}{3}L_{\Phi}\Delta\varepsilon^{-2} \ .
\end{equation}
Then Algorithm \ref{Alg:STORM-Compositional} reaches an $\varepsilon$-accurate solution $\mathbf{E}\|\nabla \Phi(\widehat{x})\|\leq \varepsilon$ with IFO complexity
$$\begin{array}{ll}
\text{IFO}&=\left(\dfrac{81M_gL_f^2H_3^{3/2
}}{\Delta}+\dfrac{81M_f^2H_2^{3/2
}}{\Delta L_g}+\dfrac{M_g^2H_1}{\frac{32}{3}L_{\Phi}\Delta K_0}432\sqrt{2(2M_g+\sqrt{H_2})H_1K_0}\right)\cdot \varepsilon^{-1}
\\
& \qquad +\left(\dfrac{864M_g^3L_f^2\sqrt{H_3}}{L_\Phi}+\dfrac{864M_f^2L_g\sqrt{H_2}}{L_\Phi}+432\sqrt{2(2M_g+\sqrt{H_2})H_1K_0}\right)\cdot \varepsilon^{-3} \ .
\end{array}$$
\end{proposition}

In practice, it may be more convenient to tune the parameters only to reach $\mathcal{O}(\ve)$-accuracy. So we have the following simple corollary, the proof of which is also left to the Supplementary Materials.

\begin{corollary}[Asymptotic IFO complexity to reach order-$\varepsilon$ accurate solution]\label{Proposition:AsymptoticEpsAccuracy}
Pick $\eta=\dfrac{1}{L_{\Phi}}$, and $a_{g}, a_{\partial g}, a_{\Phi}\sim \mathcal{O}(\varepsilon)$, $B_g, B_{\partial g}, B_f, S_g, S_{\partial g}, S_f\sim \mathcal{O}(\varepsilon^{-1})$ and $T\sim \mathcal{O}(\varepsilon^{-2})$, then for the output $\widehat{x}$ in Algorithm \ref{Alg:STORM-Compositional}, we have $\mathbf{E}\|\nabla \Phi(\widehat{x})\|\lesssim \mathcal{O}(\varepsilon)$ within IFO complexity $\lesssim \mathcal{O}(\varepsilon^{-3})$.
\end{corollary}

\section{Numerical Experiments}
\label{Sec:Experiment}

\begin{figure}
\captionsetup{margin=0cm, justification=centering}
\centering
\includegraphics[height=6cm,  width=0.75\textwidth]{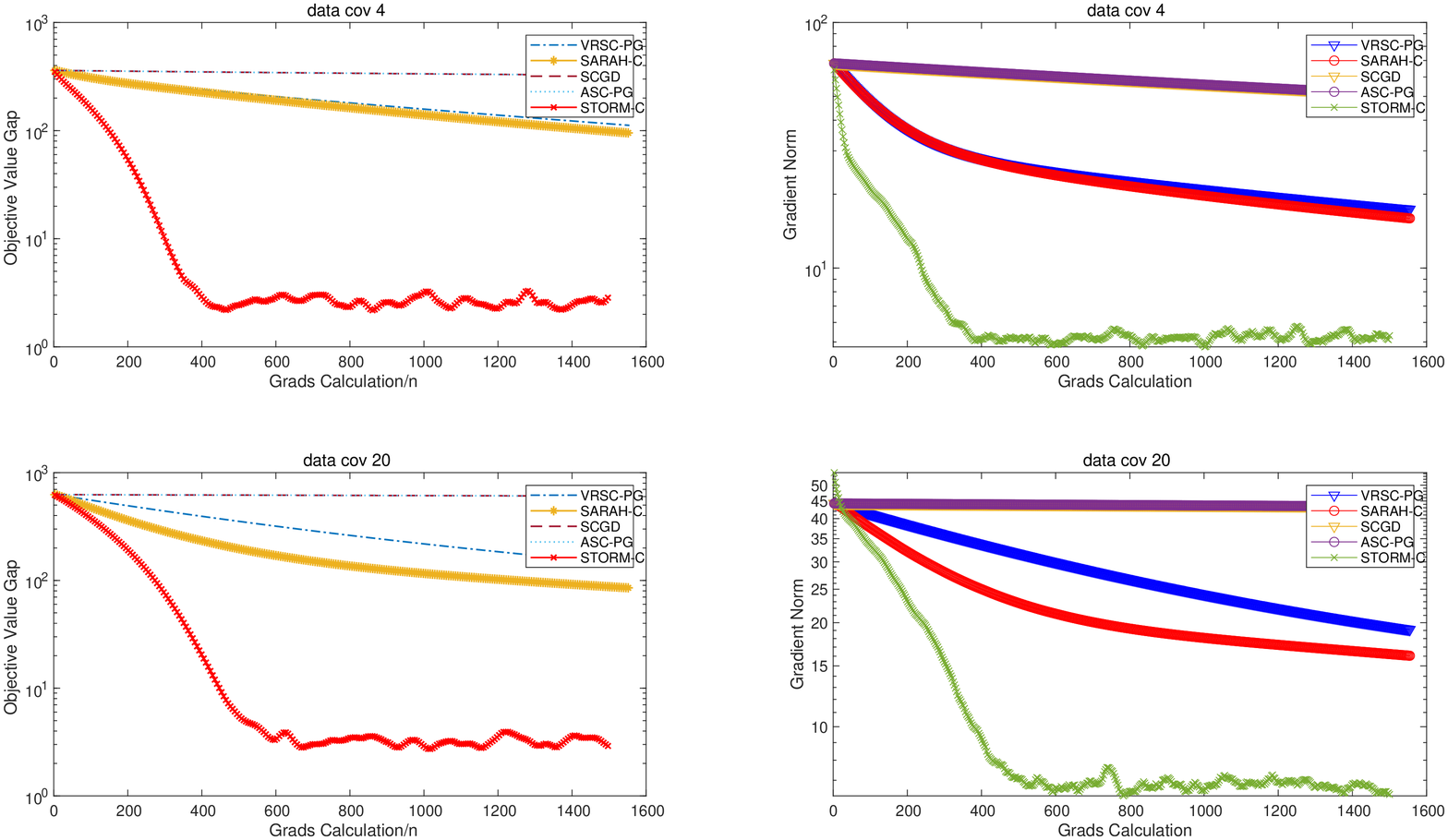}
\vspace{-0.5cm}
\caption{STORM-Compositional compared with other compositional optimization algorithms for Portfolio Management problem (finite-sum case): Left Column: Objective Function Value Gap (vertical axis) vs. Gradient Calculations (horizontal axis); Right Column: Objective Function Gradient Norm (vertical axis) vs. Gradient Calculation (horizontal axis).}
\label{Fig_port_finitesum_withreplace}
\end{figure}

We experiment here the risk-adverse portfolio management problem (see \cite{Portfolio-Optimization}, \cite{arXivSARAH-SCGD}), which can be formulated as
\vspace{-1mm}
\begin{equation}\label{Eq:PortfolioOptimization}
\min\li_{x\in \R^N} \left\{-\dfrac{1}{T}\sum\li_{t=1}^T \langle r_t, x\rangle+\dfrac{1}{T}\sum\li_{t=1}^T \left(\langle r_t, x\rangle-\dfrac{1}{T}\sum\li_{s=1}^T \langle r_s, x\rangle\right)^2\right\} \ ,
\end{equation}
where $r_t\in \R^N$ denotes the returns of $N$ assets at time $t$, and $x\in \R^N$ denotes the investment
quantity corresponding to $N$ assets. As explained in \cite{arXivSARAH-SCGD}, the example slightly does not obey our
Assumptions, but does so in a bounded domain of optimization, thus it still serves as a good example to validate our
theory. The goal is to maximize the return while controlling the
variance of the portfolio. 
%
%
%
See \cite{arXivSARAH-SCGD} for the same set-up. We compare our results with other compositional optimization algorithms, in particular the so-far claimed to be best algorithm SARAH-C (see \cite{arXivSARAH-SCGD}). To set the same standards, we developed our code based on the open source code in that work \footnote{\url{http://github.com/angeoz/SCGD}}. We compared both their finite-sum and online cases in SARAH-C by following exactly their set-up (details see \cite{arXivSARAH-SCGD}). In the finite-sum case the $r_t$'s are Gaussian with covariance matrices having condition numbers $4$ and $20$ where $T=2000, N=200$; in the online case the $r$'s are from $6$ different $25$-portfolio datasets where $T=7240, N=25$. We pick the same parameters as they used there for SARAH-C and other compositional optimization algorithms including VRSC-PG, SCGD and ASC-PG (see \cite{arXivSARAH-SCGD} for more details). In our case, we did not do much parameter tuning but we only take into consideration their relative orders suggested in Corollary \ref{Proposition:AsymptoticEpsAccuracy}. Uniformly for both choices of covariance matrices in the finite-sum case, and uniformly for all 6 families of data sets in the online case, we simply take $\ve=0.1, \eta=0.1, S_g=S_{\pt g}=S_f=B_g=B_{\pt g}=B_f=100, a_g=a_{\pt g}=a_{\Phi}=0.01$ and the results are plotted in Figures \ref{Fig_port_finitesum_withreplace} and \ref{Fig_port_OP_withreplace}, where left columns are for $\Phi(x)-\Phi^*$ and right columns are for the gradient norms as functions of IFO queries. It is seen that even for such simple and straightforward parameter setting, all data sets in each catagories (finite-sum or online) behave similarly, so that STORM-Compositional behaves at least as good as or better than (in the finite-sum case one should say much better than) SARAH-C and other compositional optimization algorithms after sufficient numbers of iterations.

\begin{figure}
\captionsetup{margin=0cm, justification=centering}
\centering
\includegraphics[height=10cm,  width=0.85\textwidth]{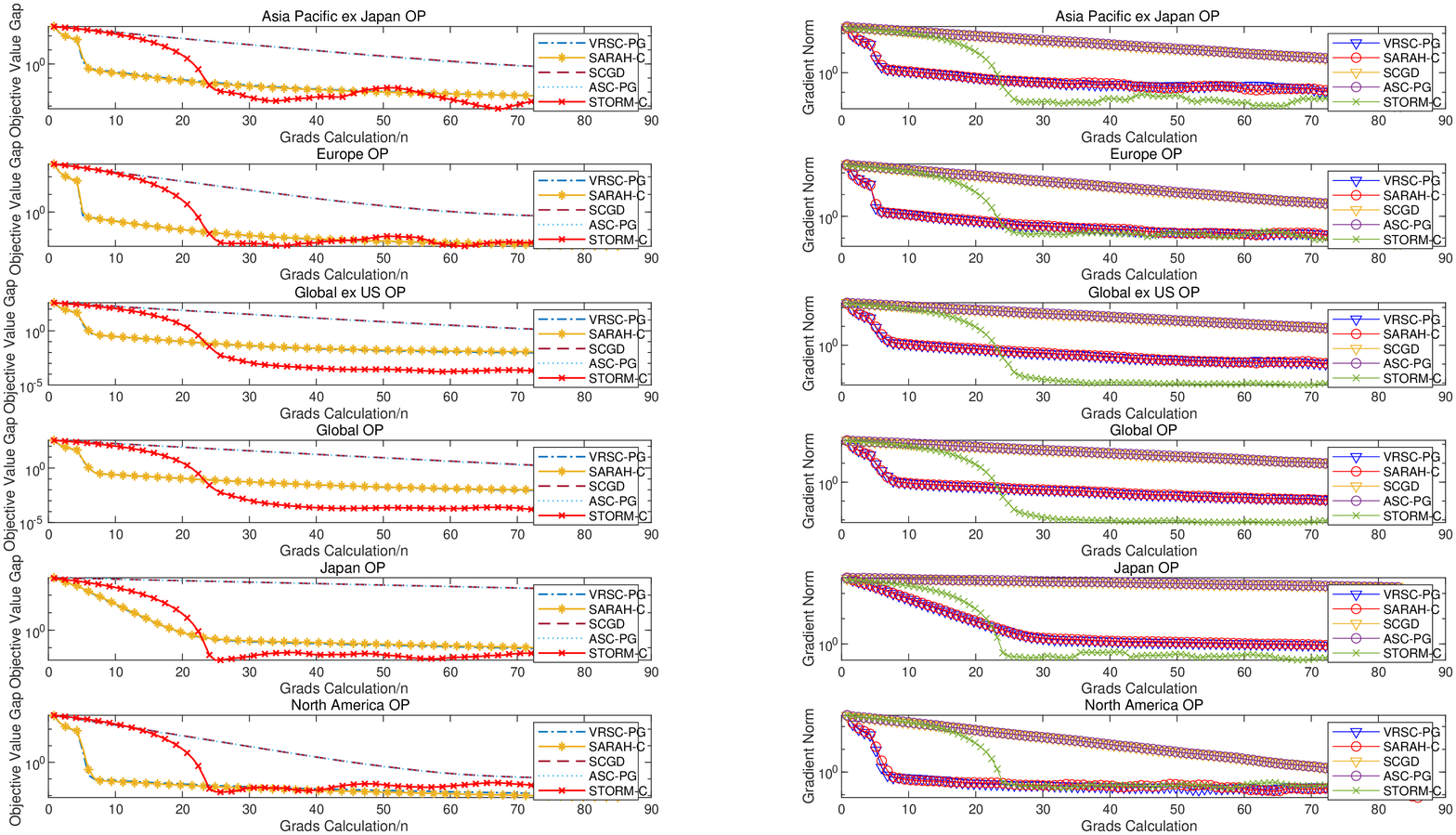}
\vspace{-0.8cm}
\caption{STORM-Compositional compared with other compositional optimization algorithms for Portfolio Management problem (online case): Left Column: Objective Function Value Gap (vertical axis) vs. Gradient Calculations (horizontal axis); Right Column: Objective Function Gradient Norm (vertical axis) vs. Gradient Calculation (horizontal axis).}
\label{Fig_port_OP_withreplace}
\end{figure}

\newpage

\section*{Broader Impacts}

This work proposes a novel compositional optimization method called STORM-Compositional that is a single-loop algorithm being easy to implement and simple to do parameter tuning, yet it matches the so far claimed to be best convergence rate $\cO(\ve^{-3})$ achieved by SARAH-C. This is definitely a positive outcome, which benefits many practical machine learning problems formulated by compositional optimization. From ethical point of view, this innovation will not raise direct concerns, ranging from the disappearance of traditional jobs, over responsibility for possible physical or psychological harm to human beings, to general dehumanization of human relationships and society at large.

\bibliographystyle{plain}
\bibliography{STORMSCGD_bibfile}

\newpage

\begin{center}
\textbf{\large{Supplementary Materials}}    
\end{center}
\appendix

\section{Auxiliary Lemmas}

We will first provide a general lemma regarding STORM estimation errors. This lemma is essentially a reformulation of Lemma 2 in \cite{STORM}, that was also re-developed in \cite{STORM-PolicyGradient}. We have

\begin{lemma}[General STORM Error Estimate]\label{Lemma:STORM-EstimatorGeneralError}
Set $q_i(x)$ to be a vector function in $\R^d$.
Let the function $q(x)=\dfrac{1}{n}\sum\limits_{i=1}^n q_i(x)$. Suppose that 
\begin{itemize}
\item[(1)] (finite variance) For a random index $i$ chosen uniformly randomly from $\{1,2,...,n\}$ we have $\mathbf{E}\|q_i(x)-q(x)\|^2\leq \sigma^2$;
\item[(2)] (stochastic Lipschitz) For a random index $i$ chosen uniformly randomly from $\{1,2,...,n\}$ and any $x, y$ we have $\mathbf{E}\|q_i(x)-q_i(y)\|^2\leq L^2\|x-y\|^2$.
\end{itemize} 

Consider the estimation $\boldsymbol{q}$ of quantity $q(x_t)$ via the stochastic ``with replacement" minibatch estimator $q(x_t, \mathcal{B}_t)=\dfrac{1}{B}\sum\limits_{i\in \mathcal{B}_t}q_i(x_t)$ following the recursion
$$\boldsymbol{q}_{t+1}\leftarrow (1-a_{t+1})\boldsymbol{q}_t+a_{t+1}q(x_{t+1}, \mathcal{B}_{t+1})+(1-a_{t+1})(q(x_{t+1}, \mathcal{B}_{t+1})-q(x_t, \mathcal{B}_{t+1})) \ ,$$
where $\{a_t\}_{t\geq 0}$ is a given numerical sequence, and $\{\mathcal{B}_t\}_{t\geq 0}$ is an i.i.d sequence of minibatches with common batchsize $B$, that are sampled uniformly randomly with replacement from $\{1,2,...,n\}$. 

Then we have the estimate 
\begin{equation}\label{Lemma:STORM-EstimatorGeneralError:Eq:Recursion-Error}
\mathbf{E}\|\boldsymbol{q}_{t+1}-q(x_{t+1})\|^2\leq (1-a_{t+1})^2\mathbf{E}\|\boldsymbol{q}_t-q(x_t)\|^2+\dfrac{2}{B}L^2(1-a_{t+1})^2\mathbf{E}\|x_{t+1}-x_t\|^2+\dfrac{2}{B}a_{t+1}^2\sigma^2 \ .
\end{equation}

Moreover, if $a_t=a\in (0,1)$ for all $t$, then we also have

\begin{equation}\label{Lemma:STORM-EstimatorGeneralError:Eq:ToBottom-Error}
\sum\limits_{t=0}^{T-1}\mathbf{E}\|\boldsymbol{q}_t-q(x_t)\|^2
\leq \dfrac{2}{a}\left[\dfrac{1}{B}L^2\sum\limits_{t=0}^{T-1}\mathbf{E}\|x_{t+1}-x_t\|^2+\dfrac{Ta^2\sigma^2}{B}+\mathbf{E}\|\boldsymbol{q}_0-q(x_0)\|^2\right] \ .
\end{equation}
\end{lemma}

\begin{proof}
Let $\mathcal{F}_t=\sigma(\mathcal{B}_1, ..., \mathcal{B}_t)$. We then expand
\begin{equation}\label{Lemma:STORM-EstimatorGeneralError:Eq:RecursiveExpansion:Step1}
\begin{array}{ll}
&\mathbf{E}\left[\left.\|\boldsymbol{q}_{t+1}-q(x_{t+1})\|^2\right|\mathcal{F}_t\right]
\\
=& \mathbf{E}\left[\left.\left\|(1-a_{t+1})\boldsymbol{q}_t+a_{t+1}q(x_{t+1}, \mathcal{B}_{t+1})+(1-a_{t+1})(q(x_{t+1}, \mathcal{B}_{t+1})-q(x_t, \mathcal{B}_{t+1}))-q(x_{t+1})\right\|^2\right|\mathcal{F}_t\right]
\\
=& \mathbf{E} \left[\left\|(1-a_{t+1})(\boldsymbol{q}_t-q(x_t))+a_{t+1}(q(x_{t+1}, \mathcal{B}_{t+1})-q(x_t))\right.\right.
\\
& \qquad \qquad \left.\left.\left.+(q(x_t)-q(x_{t+1}))+(1-a_{t+1})(q(x_{t+1}, \mathcal{B}_{t+1})-q(x_t, \mathcal{B}_{t+1}))\right\|^2\right|\mathcal{F}_t\right]
\\
=& \mathbf{E} \left[\left\|(1-a_{t+1})(\boldsymbol{q}_t-q(x_t))+(1-a_{t+1})(q(x_t)-q(x_{t+1}, \mathcal{B}_{t+1}))\right.\right.
\\
& \qquad \qquad \left.\left.\left.+(q(x_{t+1}, \mathcal{B}_{t+1})-q(x_{t+1}))+(1-a_{t+1})(q(x_{t+1}, \mathcal{B}_{t+1})-q(x_t, \mathcal{B}_{t+1}))\right\|^2\right|\mathcal{F}_t\right]
\\
=& \mathbf{E} \left[\left.\left\|(1-a_{t+1})(\boldsymbol{q}_t-q(x_t))+(1-a_{t+1})(q(x_t)-q(x_t, \mathcal{B}_{t+1}))+(q(x_{t+1}, \mathcal{B}_{t+1})-q(x_{t+1}))\right\|^2\right|\mathcal{F}_t\right]
\\
=& \mathbf{E} \left[\left\|(1-a_{t+1})(\boldsymbol{q}_t-q(x_t))+(1-a_{t+1})\left[(q(x_{t+1}, \mathcal{B}_{t+1})-q(x_{t}, \mathcal{B}_{t+1}))-(q(x_{t+1})-q(x_{t}))\right]\right.\right.
\\
& \qquad \qquad \left.\left.\left.+a_{t+1}(q(x_{t+1},\mathcal{B}_{t+1})-q(x_{t+1})\right\|^2\right|\mathcal{F}_t\right]
\\
\stackrel{(a)}{\leq} &(1-a_{t+1})^2\|\boldsymbol{q}_t-q(x_t)\|^2+2(1-a_{t+1})^2\mathbf{E}\left[\left.\|(q(x_{t+1}, \mathcal{B}_{t+1})-q(x_{t}, \mathcal{B}_{t+1}))-(q(x_{t+1})-q(x_{t}))\|^2\right|\mathcal{F}_t\right]
\\
& \qquad \qquad +2a_{t+1}^2\mathbf{E}\left[\left.\|q(x_{t+1},\mathcal{B}_{t+1})-q(x_{t+1})\|^2\right|\mathcal{F}_t\right] \ .
\end{array}
\end{equation}

Here (a) uses the independece of $\mathcal{B}_t$'s and unbiasedness of the stochastic minibatch estimator. We then take full expectation on both sides of \eqref{Lemma:STORM-EstimatorGeneralError:Eq:RecursiveExpansion:Step1}, and we get

\begin{equation}\label{Lemma:STORM-EstimatorGeneralError:Eq:RecursiveExpansion:Step2}
\begin{array}{ll}
&\mathbf{E}\|\boldsymbol{q}_{t+1}-q(x_{t+1})\|^2 
\\
\leq &(1-a_{t+1})^2\mathbf{E}\|\boldsymbol{q}_t-q(x_t)\|^2+2(1-a_{t+1})^2\mathbf{E}\|(q(x_{t+1}, \mathcal{B}_{t+1})-q(x_{t}, \mathcal{B}_{t+1}))-(q(x_{t+1})-q(x_{t}))\|^2
\\
& \qquad \qquad +2a_{t+1}^2\mathbf{E}\|q(x_{t+1},\mathcal{B}_{t+1})-q(x_{t+1})\|^2
\\
\stackrel{(b)}{\leq}&
(1-a_{t+1})^2\mathbf{E}\|\boldsymbol{q}_t-q(x_t)\|^2+2(1-a_{t+1})^2\mathbf{E}\|q(x_{t+1}, \mathcal{B}_{t+1})-q(x_{t}, \mathcal{B}_{t+1})\|^2
\\
& \qquad \qquad +2a_{t+1}^2\mathbf{E}\|q(x_{t+1},\mathcal{B}_{t+1})-q(x_{t+1})\|^2
\\
\stackrel{(c)}{\leq}& (1-a_{t+1})^2\mathbf{E}\|\boldsymbol{q}_t-q(x_t)\|^2+2\dfrac{1}{B}L^2(1-a_{t+1})^2\mathbf{E}\|x_{t+1}-x_t\|^2+2\dfrac{1}{B}a_{t+1}^2\sigma^2 \ ,
\end{array}
\end{equation}
which is just \eqref{Lemma:STORM-EstimatorGeneralError:Eq:Recursion-Error}. Here the estimate (b) uses $\mathbf{E}\|X-\mathbf{E} X\|^2\leq \mathbf{E}\|X\|^2$; the estimate (c) uses the finite variance, stochastic Lipschitz conditions and the fact that the minibatch $\mathcal{B}$ of size $B$ from $\{1,2,...,n\}$ if sampled uniformly randomly from $\{1,2,...,n\}$ \textit{with replacement}, so that

$$\begin{array}{ll}
\mathbf{E}\|q(x_{t+1}, \mathcal{B}_{t+1})-q(x_{t+1})\|^2 &=\mathbf{E}\left\|\dfrac{1}{B}\sum\limits_{i\in \mathcal{B}_{t+1}}(q_i(x_{t+1})-q(x_{t+1}))\right\|^2
\\
&=\dfrac{1}{B^2}\mathbf{E}\sum\limits_{i\in \mathcal{B}_{t+1}}\left\|q_i(x_{t+1})-q(x_{t+1})\right\|^2
\\
& = \dfrac{1}{B}\mathbf{E}\|q_i(x_{t+1})-q(x_{t+1})\|^2
\leq \dfrac{\sigma^2}{B} \ ,
\end{array}$$
and

$$\begin{array}{ll}
\mathbf{E}\|q(x_{t+1}, \mathcal{B}_{t+1})-q(x_t, \mathcal{B}_{t+1})\|^2 &=\mathbf{E}\left\|\dfrac{1}{B}\sum\limits_{i\in \mathcal{B}_{t+1}}(q_i(x_{t+1})-q_i(x_t))\right\|^2
\\
&=\dfrac{1}{B^2}\mathbf{E}\sum\limits_{i\in \mathcal{B}_{t+1}}\left\|(q_i(x_{t+1})-q_i(x_t))\right\|^2
\\
& = \dfrac{1}{B}\mathbf{E}\left\|q_i(x_{t+1})-q_i(x_t)\right\|^2
\leq  
\dfrac{L^2}{B}\mathbf{E}\|x_{t+1}-x_t\|^2 \ .
\end{array}$$

If $a_t=a\in (0,1)$ for all $t$, we can apply \eqref{Lemma:STORM-EstimatorGeneralError:Eq:Recursion-Error} recursively for $t=0, ..., T-1$ and obtain that

\begin{equation}\label{Lemma:STORM-EstimatorGeneralError:Eq:ToBottomError-1}
\sum\limits_{t=1}^T \mathbf{E}\|\boldsymbol{q}_{t}-q(x_t)\|^2  \leq (1-a)^2\sum\limits_{t=0}^{T-1}\mathbf{E}\|\boldsymbol{q}_{t}-q(x_t)\|^2+\dfrac{2}{B}L^2(1-a)^2\sum\limits_{t=0}^{T-1}\mathbf{E}\|x_{t+1}-x_t\|^2+2\dfrac{Ta^2\sigma^2}{B} \ .
\end{equation}

This gives

$$\begin{array}{ll}
&a\sum\limits_{t=0}^{T-1}\mathbf{E}\|\boldsymbol{q}_t-q(x_t)\|^2
\\
=&\sum\limits_{t=0}^{T-1}\mathbf{E}\|\boldsymbol{q}_t-q(x_t)\|^2-(1-a)\sum\limits_{t=0}^{T-1}\mathbf{E}\|\boldsymbol{q}_t-q(x_t)\|^2
\\
\stackrel{(a)}{\leq} & 
 \sum\limits_{t=1}^{T}\mathbf{E}\|\boldsymbol{q}_t-q(x_t)\|^2-(1-a)^2\sum\limits_{t=0}^{T-1}\mathbf{E}\|\boldsymbol{q}_t-q(x_t)\|^2+[\mathbf{E}\|q_0-q(x_0)\|^2-\mathbf{E}\|q_T-q(x_T)\|^2]
\\
\stackrel{(b)}{\leq} &
\dfrac{2}{B}L^2(1-a)^2\sum\limits_{t=0}^{T-1}\mathbf{E}\|x_{t+1}-x_t\|^2+2\dfrac{Ta^2\sigma^2}{B}+2\mathbf{E}\|q_0-q(x_0)\|^2\\
\stackrel{(c)}{\leq} &
\dfrac{2}{B}L^2\sum\limits_{t=0}^{T-1}\mathbf{E}\|x_{t+1}-x_t\|^2+2\dfrac{Ta^2\sigma^2}{B}+2\mathbf{E}\|q_0-q(x_0)\|^2\end{array}$$
Here in (a) we used $(1-a)\geq (1-a)^2$; in (b) we used \eqref{Lemma:STORM-EstimatorGeneralError:Eq:ToBottomError-1} and in (c) we used $a\in (0,1)$. Dividing by $a$ on both sides we obtain \eqref{Lemma:STORM-EstimatorGeneralError:Eq:ToBottom-Error}.
\end{proof}

Notice that Lemma \ref{Lemma:STORM-EstimatorGeneralError} can also be applied to the case when the $q$'s are matrices, but in that case in both assumptions and conclusions we have to replace all the operator norms $\|\cdot\|$ by Frobenius norms $\|\cdot\|_F$ correspondingly, since this is needed to conclude the last step in \eqref{Lemma:STORM-EstimatorGeneralError:Eq:RecursiveExpansion:Step1} as well as the fact that $\mathbf{E}\|q(x_{t+1}, \mathcal{B}_{t+1})-q(x_{t+1})\|_F^2
\leq \dfrac{\sigma^2}{B}$ and $\mathbf{E}\|q(x_{t+1}, \mathcal{B}_{t+1})-q(x_t, \mathcal{B}_{t+1})\|_F^2 \leq 
\dfrac{L^2}{B}\mathbf{E}\|x_{t+1}-x_t\|^2$. So we have the following corresponding assumptions and inequalities for matrix-valued $q$: 

Assume

\begin{itemize}
\item[(1)] (finite variance) For a random index $i$ chosen uniformly randomly from $\{1,2,...,n\}$ we have $\mathbf{E}\|q_i(x)-q(x)\|_F^2\leq \sigma^2$;
\item[(2)] (stochastic Lipschitz) For a random index $i$ chosen uniformly randomly from $\{1,2,...,n\}$ and any $x, y$ we have $\mathbf{E}\|q_i(x)-q_i(y)\|_F^2\leq L^2\|x-y\|^2$.
\end{itemize} 

Then conclude
\begin{itemize}
\item[(a)]
$$
\mathbf{E}\|\boldsymbol{q}_{t+1}-q(x_{t+1})\|_F^2\leq (1-a_{t+1})^2\mathbf{E}\|\boldsymbol{q}_t-q(x_t)\|_F^2+\dfrac{2}{B}L^2(1-a_{t+1})^2\mathbf{E}\|x_{t+1}-x_t\|^2+\dfrac{2}{B}a_{t+1}^2\sigma^2 \ ;
$$
\item[(b)]
$$\sum\limits_{t=0}^{T-1}\mathbf{E}\|\boldsymbol{q}_t-q(x_t)\|_F^2
\leq \dfrac{2}{a}\left[\dfrac{1}{B}L^2\sum\limits_{t=0}^{T-1}\mathbf{E}\|x_{t+1}-x_t\|^2+\dfrac{Ta^2\sigma^2}{B}+\mathbf{E}\|\boldsymbol{q}_0-q(x_0)\|_F^2\right] \ .
$$
\end{itemize}

These are indeed why we need $\|\cdot\|_F$ type assumptions in our Assumptions \ref{Assumption:Smoothness} and \ref{Assumption:FiniteVariance}. Below we will explicitly distinguish $\|A\|$ and $\|A\|_F$ when $A$ is a matrix. Since indeed $\|A\|\leq \|A\|_F$, our assumptions still lead us to the valid proof. With all these, Lemma \ref{Lemma:STORM-EstimatorGeneralError} can be directly applied to $q=g$ and $q=\partial g$, so that we obtain 

\begin{corollary}[Error Estimate for $\boldsymbol{g}$ and $\boldsymbol{G}$ sequences]\label{Corollary:ErrorEstimateToBottom:g-partialg}
We have

\begin{equation}\label{Corollary:ErrorEstimateToBottom:Eq:Recursive-g}
\mathbf{E}\|\boldsymbol{g}_{t+1}-g(x_{t+1})\|^2\leq 
(1-a_{g})^2\mathbf{E}\|\boldsymbol{g}_t-g(x_t)\|^2+\dfrac{2}{B_{g}}M_g^2\mathbf{E}\|x_{t+1}-x_t\|^2+\dfrac{2}{B_{g}}(a_{g})^2H_3 \ ,
\end{equation}

\begin{equation}\label{Corollary:ErrorEstimateToBottom:Eq:Recursive-partialg}
\mathbf{E}\|\boldsymbol{G}_{t+1}-\partial g(x_{t+1})\|_F^2\leq 
(1-a_{\partial g})^2\mathbf{E}\|\boldsymbol{G}_t-\partial g(x_t)\|_F^2+\dfrac{2}{B_{\partial g}}L_g^2\mathbf{E}\|x_{t+1}-x_t\|^2+\dfrac{2}{B_{\partial g}}(a_{\partial g})^2H_2 \ ,
\end{equation}

\begin{equation}\label{Corollary:ErrorEstimateToBottom:Eq:g}
\sum\limits_{t=0}^{T-1}\mathbf{E}\|\boldsymbol{g}_t-g(x_t)\|^2
\leq\dfrac{2}{a_g}\left[\dfrac{1}{B_g}M_g^2\sum\limits_{t=0}^{T-1}\mathbf{E}\|x_{t+1}-x_t\|^2+\dfrac{Ta_g^2H_3}{B_g}+\mathbf{E}\|\boldsymbol{g}_0-g(x_0)\|^2\right] \ ,
\end{equation}

\begin{equation}\label{Corollary:ErrorEstimateToBottom:Eq:partialg}
\sum\limits_{t=0}^{T-1}\mathbf{E}\|\boldsymbol{G}_t-\partial g(x_t)\|_F^2
\leq\dfrac{2}{a_{\partial g}}\left[\dfrac{1}{B_{\partial g}}L_g^2\sum\limits_{t=0}^{T-1}\mathbf{E}\|x_{t+1}-x_t\|^2+\dfrac{Ta_{\partial g}^2H_2}{B_{\partial g}}+\mathbf{E}\|\boldsymbol{G}_0-\partial g(x_0)\|_F^2\right] \ .
\end{equation}

\end{corollary}

We notice that \eqref{Corollary:ErrorEstimateToBottom:Eq:Recursive-g} and \eqref{Corollary:ErrorEstimateToBottom:Eq:g} also give us an estimate for $\sum\limits_{t=0}^{T-1}\mathbf{E}\|\boldsymbol{g}_{t+1}-\boldsymbol{g}_t\|^2$:

\begin{lemma}[Incremental Estimate for $\boldsymbol{g}$ sequence]\label{Lemma:Increment-boldg} We have
\begin{equation}\label{Lemma:Increment-boldg:Eq:Estimate}
\begin{array}{ll}
&\sum\limits_{t=0}^{T-1}\mathbf{E}\|\boldsymbol{g}_{t+1}-\boldsymbol{g}_t\|^2
\\
\leq &
\left[\dfrac{12M_g^2}{a_gB_g}+\dfrac{6M_g^2}{B_g}+3M_g^2\right]\sum\limits_{t=0}^{T-1}\mathbf{E}\|x_{t+1}-x_t\|^2
\\
& \qquad 
+\dfrac{12Ta_gH_3}{B_g}+\dfrac{12}{a_g}\mathbf{E}\|\boldsymbol{g}_0-g(x_0)\|^2 +\dfrac{6T}{B_g}(a_g)^2H_3 \ .
\end{array}
\end{equation} 
\end{lemma}

\begin{proof}
We can easily estimate, using \eqref{Corollary:ErrorEstimateToBottom:Eq:Recursive-g} and the inequality that $\mathbf{E}\|a+b+c\|^2\leq 3(\mathbf{E}\|a\|^2+\mathbf{E}\|b\|^2+\mathbf{E}\|c\|^2)$, as well as $0\leq 1-a_{g}\leq 1$, the following
\begin{equation}\label{Lemma:Increment-boldg:Eq:Step1}
\begin{array}{ll}
&\mathbf{E}\|\boldsymbol{g}_{t+1}-\boldsymbol{g}_t\|^2
\\
=&\mathbf{E}\|(\boldsymbol{g}_{t+1}-g(x_{t+1}))+(g(x_{t+1})-g(x_t))+(g(x_t)-\boldsymbol{g}_t)\|^2
\\
\leq & 3\mathbf{E}\|\boldsymbol{g}_{t+1}-g(x_{t+1})\|^2+3\mathbf{E}\|g(x_{t+1})-g(x_t)\|^2+3\mathbf{E}\|g(x_t)-\boldsymbol{g}_t\|^2
\\
\leq & 3[(1-a_g)^2+1]\mathbf{E}\|\boldsymbol{g}_t-g(x_t)\|^2+\left[\dfrac{6}{B_g}M_g^2(1-a_g)^2+3M_g^2\right]\mathbf{E}\|x_{t+1}-x_t\|^2+\dfrac{6}{B_g}(a_g)^2H_3 
\\
\leq & 6\mathbf{E}\|\boldsymbol{g}_t-g(x_t)\|^2+3M_g^2\left[\dfrac{2}{B_g}+1\right]\mathbf{E}\|x_{t+1}-x_t\|^2+\dfrac{6}{B_g}(a_g)^2H_3 \ .
\end{array}
\end{equation}
Summing from $t=0$ to $t=T-1$, and making use of \eqref{Corollary:ErrorEstimateToBottom:Eq:g} we obtain
$$\begin{array}{ll}
&\sum\limits_{t=0}^{T-1}\mathbf{E}\|\boldsymbol{g}_{t+1}-\boldsymbol{g}_t\|^2
\\
\leq &
\dfrac{12}{a_g}\left[\dfrac{1}{B_g}M_g^2\sum\limits_{t=0}^{T-1}\mathbf{E}\|x_{t+1}-x_t\|^2+\dfrac{Ta_g^2H_3}{B_g}+\mathbf{E}\|\boldsymbol{g}_0-g(x_0)\|^2\right] 
\\
& \qquad + 3\left[\dfrac{2M_g^2}{B_g}+M_g^2\right]\sum\limits_{t=0}^{T-1}\mathbf{E}\|x_{t+1}-x_t\|^2+\dfrac{6T}{B_g}(a_g)^2H_3 \ ,
\end{array}$$ which is \eqref{Lemma:Increment-boldg:Eq:Estimate}.
\end{proof}

Similarly, \eqref{Corollary:ErrorEstimateToBottom:Eq:Recursive-partialg} and \eqref{Corollary:ErrorEstimateToBottom:Eq:partialg} give us an estimate for $\sum\limits_{t=0}^{T-1}\mathbf{E}\|\boldsymbol{G}_{t+1}-\boldsymbol{G}_t\|_F^2$:

\begin{lemma}[Incremental Estimate for the $\boldsymbol{G}$ sequence]\label{Lemma:Increment-boldG}
We have
\begin{equation}\label{Lemma:Increment-boldG:Eq:Estimate}
\begin{array}{ll}
&\sum\limits_{t=0}^{T-1}\mathbf{E}\|\boldsymbol{G}_{t+1}-\boldsymbol{G}_t\|_F^2
\\
\leq &
\left[\dfrac{12L_g^2}{a_{\partial g}B_{\partial g}}+\dfrac{6L_g^2}{B_{\partial g}}+3L_g^2\right]\sum\limits_{t=0}^{T-1}\mathbf{E}\|x_{t+1}-x_t\|^2
\\
& \qquad 
+\dfrac{12Ta_{\partial g}H_2}{B_{\partial g}}+\dfrac{12}{a_{\partial g}}\mathbf{E}\|\boldsymbol{G}_0-\partial g(x_0)\|_F^2 +\dfrac{6T}{B_{\partial g}}(a_{\partial g})^2H_2 \ .
\end{array}
\end{equation}
\end{lemma}

\begin{proof}
We can easily estimate, using \eqref{Corollary:ErrorEstimateToBottom:Eq:Recursive-partialg} and the inequality that $\mathbf{E}\|a+b+c\|^2\leq 3(\mathbf{E}\|a\|^2+\mathbf{E}\|b\|^2+\mathbf{E}\|c\|^2)$, as well as $1-a_{\partial g}\leq 1$, the following
\begin{equation}\label{Lemma:Increment-boldG:Eq:Step1}
\begin{array}{ll}
&\mathbf{E}\|\boldsymbol{G}_{t+1}-\boldsymbol{G}_t\|_F^2
\\
=&\mathbf{E}\|(\boldsymbol{G}_{t+1}-\partial g(x_{t+1}))+(\partial g(x_{t+1})-\partial g(x_t))+(\partial g(x_t)-\boldsymbol{G}_t)\|_F^2
\\
\leq & 3\mathbf{E}\|\boldsymbol{G}_{t+1}-\partial g(x_{t+1})\|_F^2+3\mathbf{E}\|\partial g(x_{t+1})-\partial g(x_t)\|_F^2+3\mathbf{E}\|\partial g(x_t)-\boldsymbol{G}_t\|_F^2
\\
\leq & 3[(1-a_{\partial g})^2+1]\mathbf{E}\|\boldsymbol{G}_t-\partial g(x_t)\|_F^2+\left[\dfrac{6}{B_{\partial g}}L_g^2(1-a_{\partial g})^2+3L_g^2\right]\mathbf{E}\|x_{t+1}-x_t\|^2+\dfrac{6}{B_{\partial g}}(a_{\partial g})^2H_2 
\\
\leq & 6\mathbf{E}\|\boldsymbol{G}_t-\partial g(x_t)\|_F^2+3\left[\dfrac{2L_g^2}{B_{\partial g}}+L_g^2\right]\mathbf{E}\|x_{t+1}-x_t\|^2+\dfrac{6}{B_{\partial g}}(a_{\partial g})^2H_2 \ .
\end{array}
\end{equation}
Summing from $t=0$ to $t=T-1$ and making use of \eqref{Corollary:ErrorEstimateToBottom:Eq:partialg} we obtain

$$\begin{array}{ll}
&\sum\limits_{t=0}^{T-1}\mathbf{E}\|\boldsymbol{G}_{t+1}-\boldsymbol{G}_t\|_F^2
\\
\leq &
\dfrac{12}{a_{\partial g}}\left[\dfrac{1}{B_{\partial g}}L_g^2\sum\limits_{t=0}^{T-1}\mathbf{E}\|x_{t+1}-x_t\|^2+\dfrac{Ta_{\partial g}^2H_2}{B_{\partial g}}+\mathbf{E}\|\boldsymbol{G}_0-\partial g(x_0)\|_F^2\right] 
\\
& \qquad + 3\left[\dfrac{2L_g^2}{B_{\partial g}}+L_g^2\right]\sum\limits_{t=0}^{T-1}\mathbf{E}\|x_{t+1}-x_t\|^2+\dfrac{6T}{B_{\partial g}}(a_{\partial g})^2H_2 \ ,
\end{array}$$ which is \eqref{Lemma:Increment-boldG:Eq:Estimate}. 
\end{proof}

To obtain recursive estimates of similar type for $\boldsymbol{F}$ sequence, a little bit more involved technicalities are necessary.

We shall first provide an auxiliary lemma on the estimate of $\|\boldsymbol{G}_t\|$  (notice that this is not Frobenius norm). This is done through the following two lemmas. The first one is an easy incremental estimate of the $x$-sequence that will also be used later.

\begin{lemma}[Incremental Estimate for the $x$-sequence]\label{Lemma:Incremental-x}
We have, for all $t\geq 0$, that 
\begin{equation}\label{Lemma:Incremental-x:Eq:Estimate}
\|x_{t+1}-x_t\|\leq \eta\varepsilon \ .
\end{equation}
\end{lemma}

\begin{proof}
By our STORM-Compositional Algorithm 1 we know that $x_{t+1}-x_t=-\gamma_t(\widetilde{x}_{t+1}-x_t)$. Since $\gamma_t\leq \dfrac{\eta\varepsilon}{\|\widetilde{x}_{t+1}-x_t\|}$, we see that $\|x_{t+1}-x_t\|\leq \eta \varepsilon$, as desired.
\end{proof}

Now we can estimate $\|\boldsymbol{G}_t\|$ uniformly.

\begin{lemma}[Uniform bound of the norm of  $\boldsymbol{G}$-sequence]\label{Lemma:UniformBoundboldG}
When $a_{\partial g}\in (0,1)$, for all $t\geq 0$ we have
\begin{equation}\label{Lemma:UniformBoundboldG:Eq:Bound}
\|\boldsymbol{G}_t\|< 2M_g+\dfrac{L_g\eta\varepsilon}{a_{\partial g}} \ .
\end{equation}
\end{lemma}

\begin{proof}
It is easy to see from the $\boldsymbol{G}$-iteration in Algorithm 1 that we have
$$\begin{array}{ll}
\|\boldsymbol{G}_{t+1}\|& \stackrel{(a)}{\leq} 
(1-a_{\partial g})\|\boldsymbol{G}_t\|+
a_{\partial g}\|\partial g(x_{t+1}, \mathcal{B}_{t+1}^{\partial g})\|+(1-a_{\partial g})\|\partial g(x_{t+1}, \mathcal{B}_{t+1}^{\partial g})-\partial g(x_t, \mathcal{B}_{t+1}^{\partial g})\|
\\
&=(1-a_{\partial g})\|\boldsymbol{G}_t\|+
a_{\partial g}\left\|\dfrac{1}{B_{\partial g}}\sum\limits_{i\in \mathcal{B}_{t+1}^{\partial g}}\partial g_i(x_{t+1})\right\|+(1-a_{\partial g})\left\|\dfrac{1}{B_{\partial g}}\sum\limits_{i\in \mathcal{B}_{t+1}^{\partial g}}(\partial g_i(x_{t+1})-\partial g_i(x_t))\right\|
\\
\\
&\leq (1-a_{\partial g})\|\boldsymbol{G}_t\|+
a_{\partial g}\left\|\dfrac{1}{B_{\partial g}}\sum\limits_{i\in \mathcal{B}_{t+1}^{\partial g}}\partial g_i(x_{t+1})\right\|+(1-a_{\partial g})\left\|\dfrac{1}{B_{\partial g}}\sum\limits_{i\in \mathcal{B}_{t+1}^{\partial g}}(\partial g_i(x_{t+1})-\partial g_i(x_t))\right\|_F
\\
&\stackrel{(b)}{\leq} (1-a_{\partial g})\|\boldsymbol{G}_t\|+a_{\partial g}M_g+(1-a_{\partial g})L_g\|x_{t+1}-x_t\|
\\
&\stackrel{(c)}{\leq} (1-a_{\partial g})\|\boldsymbol{G}_t\|+a_{\partial g}M_g+(1-a_{\partial g})L_g\eta\varepsilon \ .
\end{array}$$
Here in (a) we used triangle inequality; in (b) we used Assumptions \ref{Assumption:Smoothness} and \ref{Assumption:Boundedness}; in (c) we used Lemma \ref{Lemma:Incremental-x}. 

Thus we have
$$\begin{array}{ll}
\|\boldsymbol{G}_{t}\|&\leq
(1-a_{\partial g})\|\boldsymbol{G}_{t-1}\|+a_{\partial g}M_g+(1-a_{\partial g})L_g\eta\varepsilon
\\
&\leq  (1-a_{\partial g})^2\|\boldsymbol{G}_{t-2}\|+a_{\partial g}M_g(1+(1-a_{\partial g}))+(1-a_{\partial g})L_g\eta \varepsilon(1+(1-a_{\partial g}))
\\
&\leq ...
\\
&\leq (1-a_{\partial g})^t\|\boldsymbol{G}_{0}\|+a_{\partial g}M_g(1+(1-a_{\partial g})+...+(1-a_{\partial g})^{t-1})
\\
& \qquad + (1-a_{\partial g})L_g\eta\varepsilon(1+(1-a_{\partial g})+...+(1-a_{\partial g})^{t-1})
\\
&< (1-a_{\partial g})^t\left\|\dfrac{1}{S_g}\sum\limits_{i\in \mathcal{S}_0^g}\partial g_i(x_0)\right\|+a_{\partial g}M_g\dfrac{1}{1-(1-a_{\partial g})}+(1-a_{\partial g})L_g\eta \varepsilon\dfrac{1}{1-(1-a_{\partial g})}
\\
&\stackrel{(a)}{<}2M_g+\dfrac{L_g\eta \varepsilon}{a_{\partial g}} \ ,
\end{array}$$
as desired. Here in (a) we used Assumption \ref{Assumption:Boundedness}.
\end{proof}

Let us provide a recursive estimate of $\|\boldsymbol{F}_t\|$.

\begin{lemma}\label{Lemma:ErrorEstimateRecursiveExpansion:gradPhi}
We have
\begin{equation}\label{Lemma:ErrorEstimateRecursiveExpansion:gradPhi:Eq:gradPhi}
\begin{array}{ll}
& \mathbf{E}\|\boldsymbol{F}_{t+1}-(\boldsymbol{G}_{t+1})^T\nabla  f(\boldsymbol{g}_{t+1})\|^2
\\ 
\leq & (1-a_{\Phi})^2\mathbf{E}\|\boldsymbol{F}_t-(\boldsymbol{G}_t)^T\nabla f(\boldsymbol{g}_{t})\|^2
\\
& +\dfrac{4}{B_f}(1-a_{\Phi})^2\left[M_f^2\mathbf{E}\|\boldsymbol{G}_{t+1}-\boldsymbol{G}_t\|_F^2+L_f^2\left(2M_g+\dfrac{L_g\eta\varepsilon}{a_{\partial g}}\right)^2\mathbf{E}\|\boldsymbol{g}_{t+1}-\boldsymbol{g}_t\|^2\right]\\
&  +\dfrac{2}{B_f}(a_{\Phi})^2\left(2M_g+\dfrac{L_g\eta\varepsilon}{a_{\partial g}}\right)^2
H_1 \ .
\end{array}
\end{equation}
\end{lemma}

\begin{proof}
First we have

\begin{equation}\label{Lemma:ErrorEstimateRecursiveExpansion:gradPhi:Eq:RecursiveExpansion:Step1}
\begin{array}{ll}
& \mathbf{E}\left[\left.\|\boldsymbol{F}_{t+1}-(\boldsymbol{G}_{t+1})^T\nabla f(\boldsymbol{g}_{t+1})\|^2\right|\mathcal{F}_t\right]
\\
=&
\mathbf{E}\left[\|(1-a_{\Phi})\boldsymbol{F}_t+a_{\Phi}(\boldsymbol{G}_{t+1})^T\nabla f(\boldsymbol{g}_{t+1}, \mathcal{B}_{t+1}^f)\right.
\\
& \  \left.\left.+(1-a_{\Phi})[(\boldsymbol{G}_{t+1})^T\nabla f(\boldsymbol{g}_{t+1}, \mathcal{B}_{t+1}^f)-(\boldsymbol{G}_t)^T\nabla f(\boldsymbol{g}_t, \mathcal{B}_{t+1}^f)]-(\boldsymbol{G}_{t+1})^T\nabla f(\boldsymbol{g}_{t+1})\|^2\right|\mathcal{F}_t\right]
\\
=& \mathbf{E}\left[\|(1-a_{\Phi})(\boldsymbol{F}_t-(\boldsymbol{G}_t)^T\nabla f(\boldsymbol{g}_{t}))+(1-a_{\Phi})[(\boldsymbol{G}_t)^T\nabla f(\boldsymbol{g}_t)-(\boldsymbol{G}_t)^T\nabla f(\boldsymbol{g}_t, \mathcal{B}_{t+1}^f)]
\right.
\\
& \qquad \left.\left.+[(\boldsymbol{G}_{t+1})^T\nabla f(\boldsymbol{g}_{t+1}, \mathcal{B}_{t+1}^f)-(\boldsymbol{G}_{t+1})^T\nabla f(\boldsymbol{g}_{t+1})]\|^2\right|\mathcal{F}_t\right]
\\
\stackrel{(a)}{=}&
(1-a_{\Phi})^2\|\boldsymbol{F}_t-(\boldsymbol{G}_t)^T\nabla f(\boldsymbol{g}_{t})\|^2
\\
& +\mathbf{E}\left[\|(1-a_{\Phi})[(\boldsymbol{G}_t)^T\nabla f(\boldsymbol{g}_t)-(\boldsymbol{G}_t)^T\nabla f(\boldsymbol{g}_t, \mathcal{B}_{t+1}^f)]
\right.
\\
& \qquad  \left.\left.+[(\boldsymbol{G}_{t+1})^T\nabla f(\boldsymbol{g}_{t+1}, \mathcal{B}_{t+1}^f)-(\boldsymbol{G}_{t+1})^T\nabla f(\boldsymbol{g}_{t+1})]\|^2\right|\mathcal{F}_t\right]
\\
=& (1-a_{\Phi})^2\|\boldsymbol{F}_t-(\boldsymbol{G}_t)^T\nabla f(\boldsymbol{g}_{t})\|^2+(I)  \ ,
\end{array}
\end{equation}
where
$$\begin{array}{ll}
(I)=&\mathbf{E}\left[\|(1-a_{\Phi})(\boldsymbol{G}_t)^T(\nabla f(\boldsymbol{g}_t)-\nabla f(\boldsymbol{g}_t, \mathcal{B}_{t+1}^f))
\right.
\\
& \qquad \left.\left.+(\boldsymbol{G}_{t+1}))^T(\nabla f(\boldsymbol{g}_{t+1}, \mathcal{B}_{t+1}^f)-\nabla f(\boldsymbol{g}_{t+1}))\|^2\right|\mathcal{F}_t\right]\end{array} \ .$$

Here in (a) we have used the fact that for random variable $a\in \mathcal{F}_a$ and $\mathbf{E}[b|\mathcal{F}_a]=0$ we have $\mathbf{E}[\|a+b\|^2|\mathcal{F}_a]=a^2+\mathbf{E}[\|b\|^2|\mathcal{F}_a]$, as well as the fact that when we take expectation with respect to minibatch $\mathcal{B}^f$, we always have $\mathbf{E}[(1-a_{\Phi})[(\boldsymbol{G}_t)\nabla f(\boldsymbol{g}_t)-(\boldsymbol{G}_t)^T\nabla f(\boldsymbol{g}_t, \mathcal{B}_{t+1}^f)]+[(\boldsymbol{G}_{t+1})^T\nabla f(\boldsymbol{g}_{t+1}, \mathcal{B}_{t+1}^f)-(\boldsymbol{G}_{t+1})^T\nabla f(\boldsymbol{g}_{t+1})]|\mathcal{F}_t]=0$. 

For part $(I)$, we apply a similar argument as in \eqref{Lemma:STORM-EstimatorGeneralError:Eq:RecursiveExpansion:Step1}, so that we have 

$$
\begin{array}{ll}
&(I)
\\
=&\mathbf{E}\left[\|(1-a_{\Phi})\left\{(\boldsymbol{G}_t)^T[(\nabla f(\boldsymbol{g}_{t})-\nabla f(\boldsymbol{g}_t, \mathcal{B}_{t+1}^f)]-(\boldsymbol{G}_{t+1})^T[\nabla f(\boldsymbol{g}_{t+1})-\nabla f(\boldsymbol{g}_{t+1}, \mathcal{B}_{t+1}^f)]\right\}\right.
\\
& \qquad +a_{\Phi}(\boldsymbol{G}_{t+1})^T(\nabla f(\boldsymbol{g}_{t+1}, \mathcal{B}_{t+1}^f)-\nabla f(\boldsymbol{g}_{t+1}))\|^2|\mathcal{F}_t]
\\
\stackrel{(a)}{\leq} & 2(1-a_{\Phi})^2\mathbf{E}\left[\|\left((\boldsymbol{G}_{t+1})^T\nabla f(\boldsymbol{g}_{t+1}, \mathcal{B}_{t+1}^f)-(\boldsymbol{G}_t)^T\nabla f(\boldsymbol{g}_{t}, \mathcal{B}_{t+1}^f)\right)\right.
\\
& \qquad \qquad \qquad \qquad  \left.\left.-\left((\boldsymbol{G}_{t+1})^T\nabla f(\boldsymbol{g}_{t+1})-(\boldsymbol{G}_t)^T\nabla f(\boldsymbol{g}_t)\right)]\|^2\right|\mathcal{F}_t\right]
\\
& + 2(a_{\Phi})^2
\mathbf{E}\left[\left.\|(\boldsymbol{G}_{t+1}))^T(\nabla f(\boldsymbol{g}_{t+1}, \mathcal{B}_{t+1}^f)-\nabla f(\boldsymbol{g}_{t+1}))\|^2\right|\mathcal{F}_t\right]
\\
\stackrel{(b)}{\leq} & 2(1-a_{\Phi})^2\mathbf{E}\left[\|\left((\boldsymbol{G}_{t+1}))^T\nabla f(\boldsymbol{g}_{t+1}, \mathcal{B}_{t+1}^f)-(\boldsymbol{G}_t)^T\nabla f(\boldsymbol{g}_{t}, \mathcal{B}_{t+1}^f)\right)\right.
\\
& \qquad \qquad \qquad \qquad  \left.\left.-\left((\boldsymbol{G}_{t+1})^T\nabla f(\boldsymbol{g}_{t+1})-(\boldsymbol{G}_t)^T\nabla f(\boldsymbol{g}_t)\right)]\|^2\right|\mathcal{F}_t\right]
\\
& + 2(a_{\Phi})^2\left(2M_g+\dfrac{L_g\eta\varepsilon}{a_{\partial g}}\right)^2
\mathbf{E}\left[\left.\|\nabla f(\boldsymbol{g}_{t+1}, \mathcal{B}_{t+1}^f)-\nabla f(\boldsymbol{g}_{t+1})\|^2\right|\mathcal{F}_t\right] \ .
\end{array}
$$

Here in (a) we used the fact that for two vectors $a, b$ we have $\mathbf{E}\|a+b\|^2\leq 2(\mathbf{E}\|a\|^2+\mathbf{E}\|b\|^2)$; in (b) we used Lemma \ref{Lemma:UniformBoundboldG}.

Taking Expectation on both sides and mimicking \eqref{Lemma:STORM-EstimatorGeneralError:Eq:RecursiveExpansion:Step2}, we obtain that 
\begin{equation}\label{Lemma:ErrorEstimateRecursiveExpansion:gradPhi:Eq:RecursiveExpansion:(I)}
\begin{array}{ll}
&\mathbf{E}(I)
\\
= & 2(1-a_{\Phi})^2\mathbf{E}\left[\left((\boldsymbol{G}_{t+1})^T\nabla f(\boldsymbol{g}_{t+1}, \mathcal{B}_{t+1}^f)-(\boldsymbol{G}_t)^T\nabla f(\boldsymbol{g}_{t}, \mathcal{B}_{t+1}^f)\right)\right.
\\
& \qquad \qquad \qquad \qquad  \left.-\left((\boldsymbol{G}_{t+1})^T\nabla f(\boldsymbol{g}_{t+1})-(\boldsymbol{G}_t)^T\nabla f(\boldsymbol{g}_t)\right)]\|^2\right]
\\
& + 2(a_{\Phi})^2\left(2M_g+\dfrac{L_g\eta\varepsilon}{a_{\partial g}}\right)^2
\mathbf{E}\left[\|\nabla f(\boldsymbol{g}_{t+1}, \mathcal{B}_{t+1}^f)-\nabla f(\boldsymbol{g}_{t+1})\|^2\right]
\\
\stackrel{(a)}{\leq} &2(1-a_{\Phi})^2\mathbf{E}\left[\left\|(\boldsymbol{G}_{t+1})^T\nabla f(\boldsymbol{g}_{t+1}, \mathcal{B}_{t+1}^f)-(\boldsymbol{G}_t)^T\nabla f(\boldsymbol{g}_{t}, \mathcal{B}_{t+1}^f)\right\|^2\right]
\\
& + 2(a_{\Phi})^2\left(2M_g+\dfrac{L_g\eta\varepsilon}{a_{\partial g}}\right)^2
\mathbf{E}\left[\|\nabla f(\boldsymbol{g}_{t+1}, \mathcal{B}_{t+1}^f)-\nabla f(\boldsymbol{g}_{t+1})\|^2\right]
\\
\stackrel{(b)}{=} & \dfrac{2}{B_f}(1-a_{\Phi})^2\mathbf{E}\left[\left\|(\boldsymbol{G}_{t+1})^T\nabla f_i(\boldsymbol{g}_{t+1})-(\boldsymbol{G}_t)^T\nabla f_i(\boldsymbol{g}_{t})\right\|^2\right]
\\
& + \dfrac{2}{B_f}(a_{\Phi})^2\left(2M_g+\dfrac{L_g\eta\varepsilon}{a_{\partial g}}\right)^2
\mathbf{E}\|\nabla f_i(\boldsymbol{g}_{t+1})-\nabla f(\boldsymbol{g}_{t+1})\|^2
\\
\stackrel{(c)}{\leq} &
\dfrac{2}{B_f}(1-a_{\Phi})^2\cdot 2\left[M_f^2\mathbf{E}\|\boldsymbol{G}_{t+1}-\boldsymbol{G}_t\|_F^2+L_f^2\left(2M_g+\dfrac{L_g\eta\varepsilon}{a_{\partial g}}\right)^2\mathbf{E}\|\boldsymbol{g}_{t+1}-\boldsymbol{g}_t\|^2\right]\\
&  +\dfrac{2}{B_f}(a_{\Phi})^2\left(2M_g+\dfrac{L_g\eta\varepsilon}{a_{\partial g}}\right)^2
H_1 \ .
\end{array}
\end{equation}
Here in (a) we used the fact that $\mathbf{E}\|X-\mathbf{E} X\|^2\leq \mathbf{E}\|X\|^2$, and the fact that when sampling the minibatches $\mathcal{B}_{t+1}^f$ with replacement we have $\mathbf{E}^{\mathcal{B}_{t+1}^f}[(\boldsymbol{G}_{t+1})^T\nabla f(\boldsymbol{g}_{t+1}, \mathcal{B}_{t+1}^f)-(\boldsymbol{G}_t)^T\nabla f(\boldsymbol{g}_t, \mathcal{B}_{t+1}^f)]=(\boldsymbol{G}_{t+1})^T\nabla f(\boldsymbol{g}_{t+1})-(\boldsymbol{G}_t)^T\nabla f(\boldsymbol{g}_t)$, where the expectation $\mathbf{E}^{\mathcal{B}_{t+1}^f}$ is taken with respect to $\mathcal{B}_{t+1}^f$; in (b) we used the fact that the minibatches $\mathcal{B}^f$ are sampled with replacement; in (c) we used Assumption \ref{Assumption:Smoothness}, Assumption \ref{Assumption:FiniteVariance}, Lemma \ref{Lemma:UniformBoundboldG} and the fact $\mathbf{E}\|a+b\|^2\leq 2\mathbf{E}(\|a\|^2+\|b\|^2)$.

Putting \eqref{Lemma:ErrorEstimateRecursiveExpansion:gradPhi:Eq:RecursiveExpansion:(I)} and \eqref{Lemma:ErrorEstimateRecursiveExpansion:gradPhi:Eq:RecursiveExpansion:Step1} together, we obtain \eqref{Lemma:ErrorEstimateRecursiveExpansion:gradPhi:Eq:gradPhi}.
\end{proof}

Mimicking again \eqref{Lemma:STORM-EstimatorGeneralError:Eq:ToBottom-Error}, we can obtain the sum estimate for $\mathbf{E}\|\boldsymbol{F}_t-(\boldsymbol{G}_t)^T\nabla f(\boldsymbol{g}_t)\|^2$. We obtain

\begin{lemma}[Error Estimate for $\boldsymbol{F}$ sequence]\label{Lemma:ErrorEstimateToBottom:gradPhi}
We have
\begin{equation}\label{Lemma:ErrorEstimateToBottom:gradPhi:Eq:Error}
\begin{array}{ll}
&\sum\limits_{t=0}^{T-1}\mathbf{E}\|\boldsymbol{F}_t-(\boldsymbol{G}_t)^T\nabla f(\boldsymbol{g}_t)\|^2
\\
\leq & \dfrac{4M_f^2}{a_{\Phi}B_f}\left\{\left[\dfrac{12L_g^2}{a_{\partial g}B_{\partial g}}+\dfrac{6L_g^2}{B_{\partial g}}+3L_g^2\right]\sum\limits_{t=0}^{T-1}\mathbf{E}\|x_{t+1}-x_t\|^2 \right.
\\
& \left. \qquad \qquad 
+\dfrac{12Ta_{\partial g}H_2}{B_{\partial g}}+\dfrac{12}{a_{\partial g}}\mathbf{E}\|\boldsymbol{G}_0-\partial g(x_0)\|_F^2 +\dfrac{6T}{B_{\partial g}}(a_{\partial g})^2H_2\right\} \
\\
&+\dfrac{4L_f^2}{a_{\Phi}B_f}\left(2M_g+\dfrac{L_g\eta\varepsilon}{a_{\partial g}}\right)\left\{\left[\dfrac{12M_g^2}{a_gB_g}+\dfrac{6M_g^2}{B_g}+3M_g^2\right]\sum\limits_{t=0}^{T-1}\mathbf{E}\|x_{t+1}-x_t\|^2\right.
\\
& \left. \qquad \qquad \qquad \qquad \qquad \qquad
+\dfrac{12Ta_gH_3}{B_g}+\dfrac{12}{a_g}\mathbf{E}\|\boldsymbol{g}_0-g(x_0)\|^2 +\dfrac{6T}{B_g}(a_g)^2H_3\right\}
\\
&+\dfrac{2T}{B_f}(a_{\Phi})\left(2M_g+\dfrac{L_g\eta\varepsilon}{a_{\partial g}}\right)^2
H_1+\dfrac{1}{a_{\Phi}}\mathbf{E}\|\boldsymbol{F}_0-(\boldsymbol{G}_0)^T\nabla f(\boldsymbol{g}_0)\|^2 \ . 
\end{array}
\end{equation}
\end{lemma}

\begin{proof}
We sum \eqref{Lemma:ErrorEstimateRecursiveExpansion:gradPhi:Eq:gradPhi} from $t=0$ to $t=T-1$ and we obtain that 
\begin{equation}\label{Lemma:ErrorEstimateRecursiveExpansion:gradPhi:Eq:gradPhi-Step1}
\begin{array}{ll}
& \sum\limits_{t=1}^T \mathbf{E}\|\boldsymbol{F}_t-(\boldsymbol{G}_t)^T\nabla f(\boldsymbol{g}_t)\|
\\
\leq & (1-a_\Phi)^2\sum\limits_{t=0}^{T-1}\mathbf{E}\|\boldsymbol{F}_t-(\boldsymbol{G}_t)^T\nabla f(\boldsymbol{g}_t)\|^2
\\
& +\dfrac{4}{B_f}(1-a_{\Phi})^2\left[M_f^2\sum\limits_{t=0}^{T-1} \mathbf{E}\|\boldsymbol{G}_{t+1}-\boldsymbol{G}_t\|_F^2+L_f^2\left(2M_g+\dfrac{L_g\eta\varepsilon}{a_{\partial g}}\right)^2\sum\limits_{t=0}^{T-1}\mathbf{E}\|\boldsymbol{g}_{t+1}-\boldsymbol{g}_t\|^2\right]\\
&  +\dfrac{2T}{B_f}(a_{\Phi})^2\left(2M_g+\dfrac{L_g\eta\varepsilon}{a_{\partial g}}\right)^2
H_1 \ .
\end{array}
\end{equation} 

By the fact that $(1-a_\Phi)^2\leq 1-a_\Phi$ this gives that
$$\begin{array}{ll}
& a_\Phi\sum\limits_{t=0}^{T-1}\mathbf{E}\|\boldsymbol{F}_t-(\boldsymbol{G}_t)^T\nabla f(\boldsymbol{g}_t)\|^2
\\
=
& \sum\limits_{t=0}^{T-1}\mathbf{E}\|\boldsymbol{F}_t-(\boldsymbol{G}_t)^T\nabla f(\boldsymbol{g}_t)\|^2- (1-a_\Phi)\sum\limits_{t=0}^{T-1}\mathbf{E}\|\boldsymbol{F}_t-(\boldsymbol{G}_t)^T\nabla f(\boldsymbol{g}_t)\|^2
\\
\leq
& \sum\limits_{t=0}^{T-1}\mathbf{E}\|\boldsymbol{F}_t-(\boldsymbol{G}_t)^T\nabla f(\boldsymbol{g}_t)\|^2- (1-a_\Phi)^2\sum\limits_{t=0}^{T-1}\mathbf{E}\|\boldsymbol{F}_t-(\boldsymbol{G}_t)^T\nabla f(\boldsymbol{g}_t)\|^2
\\
=
& \sum\limits_{t=1}^{T}\mathbf{E}\|\boldsymbol{F}_t-(\boldsymbol{G}_t)^T\nabla f(\boldsymbol{g}_t)\|^2- (1-a_\Phi)^2\sum\limits_{t=0}^{T-1}\mathbf{E}\|\boldsymbol{F}_t-(\boldsymbol{G}_t)^T\nabla f(\boldsymbol{g}_t)\|^2
\\
& \qquad +\mathbf{E}\|\boldsymbol{F}_0-(\boldsymbol{G}_0)^T\nabla f(\boldsymbol{g}_0)\|^2-\mathbf{E}\|\boldsymbol{F}_T-(\boldsymbol{G}_T)^T\nabla f(\boldsymbol{g}_T)\|^2
\\
\stackrel{(a)}{\leq} 
& \dfrac{4}{B_f}(1-a_{\Phi})^2\left[M_f^2\sum\limits_{t=0}^{T-1} \mathbf{E}\|\boldsymbol{G}_{t+1}-\boldsymbol{G}_t\|_F^2+L_f^2\left(2M_g+\dfrac{L_g\eta\varepsilon}{a_{\partial g}}\right)^2\sum\limits_{t=0}^{T-1}\mathbf{E}\|\boldsymbol{g}_{t+1}-\boldsymbol{g}_t\|^2\right]\\
&  +\dfrac{2T}{B_f}(a_{\Phi})^2\left(2M_g+\dfrac{L_g\eta\varepsilon}{a_{\partial g}}\right)^2
H_1+\mathbf{E}\|\boldsymbol{F}_0-(\boldsymbol{G}_0)^T\nabla f(\boldsymbol{g}_0)\|^2 \ .
\end{array}$$

Here in (a) we used \eqref{Lemma:ErrorEstimateRecursiveExpansion:gradPhi:Eq:gradPhi-Step1}. Dividing on both sides by $a_\Phi$ and noticing that $1-a_\Phi\leq 1$ we obtain \begin{equation}\label{Lemma:ErrorEstimateToBottom:gradPhi:Eq:Error:Step1}
\begin{array}{ll}
&\sum\limits_{t=0}^{T-1}\mathbf{E}\|\boldsymbol{F}_t-(\boldsymbol{G}_t)^T\nabla f(\boldsymbol{g}_t)\|^2
\\
\leq & \dfrac{1}{a_\Phi}\left\{\dfrac{4}{B_f}\left[M_f^2\sum\limits_{t=0}^{T-1} \mathbf{E}\|\boldsymbol{G}_{t+1}-\boldsymbol{G}_t\|_F^2+L_f^2\left(2M_g+\dfrac{L_g\eta\varepsilon}{a_{\partial g}}\right)^2\sum\limits_{t=0}^{T-1}\mathbf{E}\|\boldsymbol{g}_{t+1}-\boldsymbol{g}_t\|^2\right]\right.\\
& \left. \qquad \qquad  +\dfrac{2T}{B_f}(a_{\Phi})^2\left(2M_g+\dfrac{L_g\eta\varepsilon}{a_{\partial g}}\right)^2
H_1+\mathbf{E}\|\boldsymbol{F}_0-(\boldsymbol{G}_0)^T\nabla f(\boldsymbol{g}_0)\|^2\right\} \ .
\end{array}
\end{equation}

Now we combine \eqref{Lemma:ErrorEstimateToBottom:gradPhi:Eq:Error:Step1}, \eqref{Lemma:Increment-boldg:Eq:Estimate}, \eqref{Lemma:Increment-boldG:Eq:Estimate} and we obtain 

$$
\begin{array}{ll}
&\sum\limits_{t=0}^{T-1}\mathbf{E}\|\boldsymbol{F}_t-(\boldsymbol{G}_t)^T\nabla f(\boldsymbol{g}_t)\|^2
\\
\leq & \dfrac{4M_f^2}{a_{\Phi}B_f}\left\{\left[\dfrac{12L_g^2}{a_{\partial g}B_{\partial g}}+\dfrac{6L_g^2}{B_{\partial g}}+3L_g^2\right]\sum\limits_{t=0}^{T-1}\mathbf{E}\|x_{t+1}-x_t\|^2 \right.
\\
& \left. \qquad \qquad 
+\dfrac{12Ta_{\partial g}H_2}{B_{\partial g}}+\dfrac{12}{a_{\partial g}}\mathbf{E}\|\boldsymbol{G}_0-\partial g(x_0)\|_F^2 +\dfrac{6T}{B_{\partial g}}(a_{\partial g})^2H_2\right\} \
\\
&+\dfrac{4L_f^2}{a_{\Phi}B_f}\left(2M_g+\dfrac{L_g\eta\varepsilon}{a_{\partial g}}\right)\left\{\left[\dfrac{12M_g^2}{a_gB_g}+\dfrac{6M_g^2}{B_g}+3M_g^2\right]\sum\limits_{t=0}^{T-1}\mathbf{E}\|x_{t+1}-x_t\|^2\right.
\\
& \left. \qquad \qquad \qquad \qquad \qquad \qquad
+\dfrac{12Ta_gH_3}{B_g}+\dfrac{12}{a_g}\mathbf{E}\|\boldsymbol{g}_0-g(x_0)\|^2 +\dfrac{6T}{B_g}(a_g)^2H_3\right\}
\\
&+\dfrac{2T}{B_f}(a_{\Phi})\left(2M_g+\dfrac{L_g\eta\varepsilon}{a_{\partial g}}\right)^2
H_1+\dfrac{1}{a_{\Phi}}\mathbf{E}\|\boldsymbol{F}_0-(\boldsymbol{G}_0)^T\nabla f(\boldsymbol{g}_0)\|^2 
\end{array}
$$
which is \eqref{Lemma:ErrorEstimateToBottom:gradPhi:Eq:Error}.
\end{proof}

\begin{lemma}[Error Estimate of $\mathbf{E}\|\boldsymbol{F}_t-\nabla \Phi(x_t)\|^2$]\label{Lemma:ErrorEstimateToBottom:FTogradPhi}
We have
\begin{equation}\label{Lemma:ErrorEstimateToBottom:FTogradPhi:Eq:Estimate}
\begin{array}{ll}
& \dfrac{1}{T}\sum\limits_{t=0}^{T-1}\mathbf{E}\|\boldsymbol{F}_t-\nabla \Phi(x_t)\|^2
\\
\leq &\left\{\dfrac{36M_f^2L_g^2}{a_{\Phi}B_f}\left[\dfrac{4}{a_{\partial g}B_{\partial g}}+\dfrac{2}{B_{\partial g}}+1\right]+\dfrac{36L_f^2M_g^2}{a_{\Phi}B_f}\left(2M_g+\dfrac{L_g\eta\varepsilon}{a_{\partial g}}\right)\left[\dfrac{4}{a_gB_g}+\dfrac{2}{B_g}+1\right]\right.
\\
& \qquad \left.+\dfrac{6M_f^2L_g^2}{a_{\partial g}B_{\partial g}}+\dfrac{6M_g^4L_f^2}{a_gB_g}\right\}\eta^2\varepsilon^2
\\
& +\dfrac{6M_f^2}{Ta_{\partial g}S_{\partial g}}\left(\dfrac{24}{a_\Phi B_f }+1\right)H_2
+\dfrac{6L_f^2}{Ta_gS_g}\left(\dfrac{24}{a_{\Phi}B_f}\left(2M_g+\dfrac{L_g\eta\varepsilon}{a_{\partial g}}\right)+M_g^2\right)H_3
+ \dfrac{3M_g^2}{Ta_{\Phi}S_f}H_1
\\
& +\dfrac{72M_f^2a_{\partial g}H_2}{a_{\Phi}B_fB_{\partial g}}(2+a_{\partial g})+\dfrac{72L_f^2a_gH_3}{a_{\Phi}B_fB_g}\left(2M_g+\dfrac{L_g\eta \varepsilon}{a_{\partial g}}\right)(2+a_g)
\\
& + \dfrac{6a_\Phi}{B_f}\left(2M_g+\dfrac{L_g\eta\varepsilon}{a_{\partial g}}\right)^2H_1+\dfrac{6M_f^2a_{\partial g}H_2}{B_{\partial g}}+\dfrac{6M_g^2L_f^2a_gH_3}{B_g} \ .
\end{array}
\end{equation}
\end{lemma}

\begin{proof}

Since \begin{equation}\label{Proposition:AsymptoticEpsAccuracy:Eq:EstimateF-To-gradPhi}
\begin{array}{ll}
&\sum\limits_{t=0}^{T-1}\mathbf{E}\|\boldsymbol{F}_t-\nabla \Phi(x_t)\|^2
\\
\leq &3\sum\limits_{t=0}^{T-1}\mathbf{E}\|\boldsymbol{F}_t-(\boldsymbol{G}_t)^T\nabla f(\boldsymbol{g}_t)\|^2+3\sum\limits_{t=0}^{T-1}\mathbf{E}\|(\boldsymbol{G}_t)^T\nabla f(\boldsymbol{g}_t)-(\partial g(x_t))^T\nabla f(\boldsymbol{g}_t)\|^2
\\
& \qquad +3\sum\limits_{t=0}^{T-1}\mathbf{E}\|(\partial g(x_t))^T\nabla f(\boldsymbol{g}_t)-(\partial g(x_t))^T\nabla f(g(x_t))\|^2
\\
\leq &3\sum\limits_{t=0}^{T-1}\mathbf{E}\|\boldsymbol{F}_t-(\boldsymbol{G}_t)^T\nabla f(\boldsymbol{g}_t)\|^2+3M_f^2\sum\limits_{t=0}^{T-1}\mathbf{E}\|\boldsymbol{G}_t-\partial g(x_t)\|_F^2 +3M_g^2L_f^2\sum\limits_{t=0}^{T-1}\mathbf{E}\|\boldsymbol{g}_t-g(x_t)\|^2 \ .
\end{array}
\end{equation}

Now we can combine \eqref{Lemma:ErrorEstimateToBottom:gradPhi:Eq:Error}, \eqref{Corollary:ErrorEstimateToBottom:Eq:g},  \eqref{Corollary:ErrorEstimateToBottom:Eq:partialg}, and we use \eqref{Proposition:AsymptoticEpsAccuracy:Eq:EstimateF-To-gradPhi}, so that we obtain

\begin{equation}\label{Lemma:ErrorEstimateToBottom:FTogradPhi:Eq:Error}
\begin{array}{ll}
& \sum\limits_{t=0}^{T-1}\mathbf{E}\|\boldsymbol{F}_t-\nabla \Phi(x_t)\|^2
\\
\leq & \dfrac{12M_f^2}{a_{\Phi}B_f}\left\{\left[\dfrac{12L_g^2}{a_{\partial g}B_{\partial g}}+\dfrac{6L_g^2}{B_{\partial g}}+3L_g^2\right]\sum\limits_{t=0}^{T-1}\mathbf{E}\|x_{t+1}-x_t\|^2 \right.
\\
& \left. \qquad \qquad 
+\dfrac{12Ta_{\partial g}H_2}{B_{\partial g}}+\dfrac{12}{a_{\partial g}}\mathbf{E}\|\boldsymbol{G}_0-\partial g(x_0)\|_F^2 +\dfrac{6T}{B_{\partial g}}(a_{\partial g})^2H_2\right\} \
\\
&+\dfrac{12L_f^2}{a_{\Phi}B_f}\left(2M_g+\dfrac{L_g\eta\varepsilon}{a_{\partial g}}\right)\left\{\left[\dfrac{12M_g^2}{a_gB_g}+\dfrac{6M_g^2}{B_g}+3M_g^2\right]\sum\limits_{t=0}^{T-1}\mathbf{E}\|x_{t+1}-x_t\|^2\right.
\\
& \left. \qquad \qquad \qquad \qquad \qquad \qquad
+\dfrac{12Ta_gH_3}{B_g}+\dfrac{12}{a_g}\mathbf{E}\|\boldsymbol{g}_0-g(x_0)\|^2 +\dfrac{6T}{B_g}(a_g)^2H_3\right\}
\\
&+\dfrac{6T}{B_f}(a_{\Phi})\left(2M_g+\dfrac{L_g\eta\varepsilon}{a_{\partial g}}\right)^2
H_1+\dfrac{3}{a_{\Phi}}\mathbf{E}\|\boldsymbol{F}_0-(\boldsymbol{G}_0)^T\nabla f(\boldsymbol{g}_0)\|^2 
\\
& +\dfrac{6M_f^2}{a_{\partial g}}\left[\dfrac{1}{B_{\partial g}}L_g^2\sum\limits_{t=0}^{T-1}\mathbf{E}\|x_{t+1}-x_t\|^2+\dfrac{Ta_{\partial g}^2H_2}{B_{\partial g}}+\mathbf{E}\|\boldsymbol{G}_0-\partial g(x_0)\|_F^2\right] 
\\
& + \dfrac{6M_g^2L_f^2}{a_g}\left[\dfrac{1}{B_g}M_g^2\sum\limits_{t=0}^{T-1}\mathbf{E}\|x_{t+1}-x_t\|^2+\dfrac{Ta_g^2H_3}{B_g}+\mathbf{E}\|\boldsymbol{g}_0-g(x_0)\|^2\right]
\\
=&\left\{\dfrac{36M_f^2L_g^2}{a_{\Phi}B_f}\left[\dfrac{4}{a_{\partial g}B_{\partial g}}+\dfrac{2}{B_{\partial g}}+1\right]+\dfrac{36L_f^2M_g^2}{a_{\Phi}B_f}\left(2M_g+\dfrac{L_g\eta\varepsilon}{a_{\partial g}}\right)\left[\dfrac{4}{a_gB_g}+\dfrac{2}{B_g}+1\right]\right.
\\
& \qquad \left.+\dfrac{6M_f^2L_g^2}{a_{\partial g}B_{\partial g}}+\dfrac{6M_g^4L_f^2}{a_gB_g}\right\}\sum\limits_{t=0}^{T-1}\mathbf{E}\|x_{t+1}-x_t\|^2
\\
& +\dfrac{6M_f^2}{a_{\partial g}}\left(\dfrac{24}{a_\Phi B_f }+1\right)\mathbf{E}\|\boldsymbol{G}_0-\partial g(x_0)\|_F^2
\\
& +\dfrac{6L_f^2}{a_g}\left(\dfrac{24}{a_{\Phi}B_f}\left(2M_g+\dfrac{L_g\eta\varepsilon}{a_{\partial g}}\right)+M_g^2\right)\mathbf{E}\|\boldsymbol{g}_0-g(x_0)\|^2
\\
& + \dfrac{3}{a_{\Phi}}\mathbf{E}\|\boldsymbol{F}_0-(\boldsymbol{G}_0)^T\nabla f(\boldsymbol{g}_0)\|^2
\\
& +T\left\{\dfrac{72M_f^2a_{\partial g}H_2}{a_{\Phi}B_fB_{\partial g}}(2+a_{\partial g})+\dfrac{72L_f^2a_gH_3}{a_{\Phi}B_fB_g}\left(2M_g+\dfrac{L_g\eta \varepsilon}{a_{\partial g}}\right)(2+a_g)\right.
\\
&\qquad \left. + \dfrac{6a_\Phi}{B_f}\left(2M_g+\dfrac{L_g\eta\varepsilon}{a_{\partial g}}\right)^2H_1+\dfrac{6M_f^2a_{\partial g}H_2}{B_{\partial g}}+\dfrac{6M_g^2L_f^2a_gH_3}{B_g}\right\} \ .
\end{array}
\end{equation}
We take into account that by Lemma \ref{Lemma:Incremental-x} we have $\sum\limits_{t=0}^{T-1}\mathbf{E}\|x_{t+1}-x_t\|^2\leq T\eta^2\varepsilon^2$. We also notice that since the initial batches are sampled with replacement we have $\mathbf{E}\|\boldsymbol{g}_0-g(x_0)\|^2\leq \dfrac{H_3}{S_g}$, $\mathbf{E}\|\boldsymbol{G}_0-\partial g(x_0)\|_F^2\leq \dfrac{H_2}{S_{\partial g}}$ and 
$\mathbf{E}\|\boldsymbol{F}_0-(\boldsymbol{G}_0)\nabla f(\boldsymbol{g}_0)\|^2\leq\dfrac{M_g^2H_1}{S_f}$. Taking all these into account, \eqref{Lemma:ErrorEstimateToBottom:FTogradPhi:Eq:Error} yields \eqref{Lemma:ErrorEstimateToBottom:FTogradPhi:Eq:Estimate}.
\end{proof}

\begin{remark}[The effectiveness of the $a$-parameters]\label{Remark:EffectParameter-a}
We see from the estimate \eqref{Lemma:ErrorEstimateToBottom:FTogradPhi:Eq:Estimate} that each term in the upper bound on the right hand side contains a factor related to $a_g$, $a_{\partial g}$ or $a_{\Phi}$. These parameters are used to tune the convergence without using adaptive learning rates or inner loop restart mechanisms. The use of exponential moving average estimator thus allows us to do continuous
training without restarting the iterations.
\end{remark}

\section{Proof of Proposition 3.1}

\begin{proof}
The proof contains a nice argument adapted from \cite{XiaoEtAlCompositional}.
By a standard Taylor's expansion argument we have
\begin{equation}\label{Theorem:MainTheoremConvergence:Eq:TaylorExpansion}
\begin{array}{ll}
\Phi(x_{t+1})&\leq \Phi(x_t)+(\nabla \Phi(x_t))^T(x_{t+1}-x_t)+\dfrac{L_\Phi}{2}\|x_{t+1}-x_t\|^2
\\
& \stackrel{(a)}{=} \Phi(x_t)+\dfrac{L_\Phi\gamma_t^2}{2}\|\widetilde{x}_{t+1}-x_t\|^2+\gamma_t\langle\nabla \Phi(x_t), \widetilde{x}_{t+1}-x_t\rangle
\\
&\stackrel{(b)}{=}
\Phi(x_t)+\dfrac{L_\Phi\gamma_t^2}{2}\|\widetilde{x}_{t+1}-x_t\|^2+\gamma_t\langle\boldsymbol{F}_t, \widetilde{x}_{t+1}-x_t\rangle+\gamma_t\langle\nabla \Phi(x_t)-\boldsymbol{F}_t, \widetilde{x}_{t+1}-x_t\rangle
\\
& \stackrel{(c)}{=}
\Phi(x_t)+\dfrac{L_\Phi\gamma_t^2}{2}\|\widetilde{x}_{t+1}-x_t\|^2-\dfrac{\gamma_t}{\eta}\|\widetilde{x}_{t+1}-x_t\|^2+\gamma_t\langle\nabla \Phi(x_t)-\boldsymbol{F}_t, \widetilde{x}_{t+1}-x_t\rangle
\\
& \stackrel{(d)}{\leq} \Phi(x_t)+\dfrac{L_\Phi\gamma_t^2}{2}\|\widetilde{x}_{t+1}-x_t\|^2-\dfrac{\gamma_t}{\eta}\|\widetilde{x}_{t+1}-x_t\|^2+\dfrac{L_\Phi\gamma_t^2}{2}\|\widetilde{x}_{t+1}-x_t\|^2
\\
& \qquad \qquad +\dfrac{1}{2L_{\Phi}}\|\boldsymbol{F}_t-\nabla \Phi(x_t)\|^2
\\
& = \Phi(x_t)-\left(\dfrac{\gamma_t}{\eta}-L_\Phi\gamma_t^2\right)\|\widetilde{x}_{t+1}-x_t\|^2+\dfrac{1}{2L_{\Phi}}\|\boldsymbol{F}_t-\nabla \Phi(x_t)\|^2
\\
&\stackrel{(e)}{=} \Phi(x_t)-\dfrac{1}{L_\Phi}\left(
\gamma_t-\gamma_t^2\right)\|\boldsymbol{F}_t\|^2+\dfrac{1}{2L_{\Phi}}\|\boldsymbol{F}_t-\nabla \Phi(x_t)\|^2
\\
&\stackrel{(f)}{\leq}
\Phi(x_t)-\dfrac{1}{2L_\Phi}\gamma_t\|\boldsymbol{F}_t\|^2+\dfrac{1}{2L_{\Phi}}\|\boldsymbol{F}_t-\nabla \Phi(x_t)\|^2 
\\
&\stackrel{(g)}{\leq}
\Phi(x_t)-\dfrac{\varepsilon}{2L_\Phi}\|\boldsymbol{F}_t\|+\dfrac{\varepsilon^2}{4L_\Phi}+\dfrac{1}{2L_{\Phi}}\|\boldsymbol{F}_t-\nabla \Phi(x_t)\|^2  \ .
\end{array}
\end{equation}
Here in (a), (b), (c) we used the fact that in the STORM-Compositional Algorithm 1, our main iteration for $x$ is given by $x_{t+1}=x_t+\gamma_t(\widetilde{x}_{t+1}-x_t)$, $\widetilde{x}_{t+1}=x_t-\eta\boldsymbol{F}_t$; in (d) we used the Cauchy-Schwarz inequality; in (e) we used the fact that we pick $\eta=\dfrac{1}{L_{\Phi}}$ as well as the fact that $\widetilde{x}_{t+1}=x_t-\eta \boldsymbol{F}_t$; in (f) we used the fact that in our STORM-Compositional Algorithm 1 we pick $0\leq \gamma_t\leq \dfrac{1}{2}$ and when $0\leq \gamma_t\leq \dfrac{1}{2}$ we have $\gamma_t-\gamma_t^2\geq \dfrac{1}{2}\gamma_t$; in (g) we used the fact that in our STORM-Compositional Algorithm 1 we pick $\gamma_t=\min\left\{\dfrac{\varepsilon}{\|\boldsymbol{F}_t\|}, \dfrac{1}{2}\right\}$ and that $\gamma_t\|\boldsymbol{F}_t\|^2=\varepsilon^2\left\{\|\boldsymbol{F}_t\|\varepsilon^{-1}, \dfrac{1}{2}\|\boldsymbol{F}_t\|^2\varepsilon^{-2}\right\}\geq \varepsilon\|\boldsymbol{F}_t\|-\dfrac{1}{2}\varepsilon^2$, the latter due to the elementary inequality $\min\left\{|x|, \dfrac{1}{2}x^2\right\}\geq |x|-\dfrac{1}{2}$ for all $x\in \mathbb{R}$. 

Summing \eqref{Theorem:MainTheoremConvergence:Eq:TaylorExpansion} from $t=0,1, ... ,T-1$ and taking expectation on both sides allows us to conclude that 
$$\begin{array}{ll}
\Phi^*\leq \mathbf{E} [\Phi(x_T)]& \leq \Phi(x_0)-\dfrac{\varepsilon}{2L_\Phi}\sum\limits_{t=0}^{T-1}\mathbf{E}\|\boldsymbol{F}_t\|+\dfrac{\varepsilon^2T}{4L_\Phi}+\dfrac{1}{2L_\Phi}\sum\limits_{t=0}^{T-1}\mathbf{E}\|\boldsymbol{F}_t-\nabla \Phi(x_t)\|^2 \ .
\end{array}$$

Rearranging the above and making use of Assumption \ref{Assumption:FiniteGap} we obtain that 
\begin{equation}\label{Theorem:MainTheoremConvergence:Eq:FinalEstimate:Step0}
\begin{array}{ll}
\dfrac{1}{T}\sum\limits_{t=0}^{T-1}\mathbf{E}\|\boldsymbol{F}_t\|&\leq \dfrac{2L_\Phi\Delta}{T\varepsilon}+\dfrac{\varepsilon}{2}+\dfrac{1}{\varepsilon}\left(\dfrac{1}{T}\sum\limits_{t=0}^{T-1}\mathbf{E}\|\boldsymbol{F}_t-\nabla \Phi(x_t)\|^2\right) \ .
\end{array}
\end{equation}

Since the output $\widehat{x}$ in Algorithm 1 is chosen uniformly randomly from $x_0,...,x_{T-1}$, the above further gives
\begin{equation}\label{Theorem:MainTheoremConvergence:Eq:FinalEstimate:Step1}
\begin{array}{ll}
\mathbf{E}\|\nabla \Phi(\widehat{x})\| & =\dfrac{1}{T}\sum\limits_{t=0}^{T-1}\mathbf{E}\|\nabla \Phi(x_t)\|
\\
&\stackrel{(a)}{\leq} \dfrac{1}{T}\sum\limits_{t=0}^{T-1}\left(\mathbf{E}\|\boldsymbol{F}_t\|+\mathbf{E}\|\nabla \Phi(x_t)-\boldsymbol{F}_t\|\right)
\\
&= \dfrac{1}{T}\sum\limits_{t=0}^{T-1}\mathbf{E}\|\boldsymbol{F}_t\|+\dfrac{1}{T}\sum\limits_{t=0}^{T-1}\mathbf{E}\|\nabla \Phi(x_t)-\boldsymbol{F}_t\|
\\
& = 
\dfrac{1}{T}\sum\limits_{t=0}^{T-1}\mathbf{E}\|\boldsymbol{F}_t\|+\left[\left(\dfrac{1}{T}\sum\limits_{t=0}^{T-1}\mathbf{E}\|\nabla \Phi(x_t)-\boldsymbol{F}_t\|\right)^2\right]^{1/2}
\\
& \stackrel{(b)}{\leq} 
\dfrac{1}{T}\sum\limits_{t=0}^{T-1}\mathbf{E}\|\boldsymbol{F}_t\|+\left[\dfrac{1}{T}\sum\limits_{t=0}^{T-1}\mathbf{E}\|\nabla \Phi(x_t)-\boldsymbol{F}_t\|^2\right]^{1/2} \ .
\end{array}
\end{equation} \
Here in (a) we used triangle inequality and in (b) we used the fact that $\left(\dfrac{1}{T}\sum\limits_{t=0}^{T-1}\mathbf{E}\|\nabla \Phi(x_t)-\boldsymbol{F}_t\|\right)^2\leq \sum\limits_{t=0}^{T-1}\dfrac{1}{T^2}\sum\limits_{t=0}^{T-1}\mathbf{E}\|\nabla \Phi(x_t)-\boldsymbol{F}_t\|^2=\dfrac{1}{T}\sum\limits_{t=0}^{T-1}\mathbf{E}\|\nabla \Phi(x_t)-\boldsymbol{F}_t\|^2$ by Cauchy-Schwarz inequality. Combining \eqref{Theorem:MainTheoremConvergence:Eq:FinalEstimate:Step0} and \eqref{Theorem:MainTheoremConvergence:Eq:FinalEstimate:Step1} we obtain that

\begin{equation}\label{Theorem:MainTheoremConvergence:Eq:FinalEstimate:Step2}
\begin{array}{ll}
\mathbf{E}\|\nabla \Phi(\widehat{x})\| & \leq \dfrac{2L_\Phi\Delta}{T\varepsilon}+\dfrac{\varepsilon}{2}+\dfrac{1}{\varepsilon}\left(\dfrac{1}{T}\sum\limits_{t=0}^{T-1}\mathbf{E}\|\boldsymbol{F}_t-\nabla \Phi(x_t)\|^2\right)+
\left(\dfrac{1}{T}\sum\limits_{t=0}^{T-1}\mathbf{E}\|\nabla \Phi(x_t)-\boldsymbol{F}_t\|^2\right)^{1/2} \ .
\end{array}
\end{equation}
Taking into account that $\dfrac{1}{T}\sum\limits_{t=0}^{T-1}\mathbf{E}\|\boldsymbol{F}_t-\nabla \Phi(x_t)\|^2\leq A\varepsilon^2$ by Proposition 3.1 in the main text of the paper, we proved the statement. 
\end{proof}

\section{Proof of Proposition 3.2}

\begin{proof}
We look at Theorem 1 and \eqref{Proposition:AsymptoticEpsAccuracy:Eq:EstimateF-To-gradPhi}, and we pick the $a$-parameters so that 

\begin{equation}\label{Eq:PickParameters-a-basedon-ErrorEquipartition}
\begin{array}{ll}
\dfrac{1}{T}\sum\limits_{t=0}^{T-1}\mathbf{E}\|\boldsymbol{g}_t-g(x_t)\|^2\leq \dfrac{\varepsilon^2}{9\cdot 16M_g^2L_f^2}  \ , 
\\
\dfrac{1}{T}\sum\limits_{t=0}^{T-1}\mathbf{E}\|\boldsymbol{G}_t-\partial g(x_t)\|_F^2\leq \dfrac{\varepsilon^2}{9\cdot 16M_f^2}  \ , 
\\
\dfrac{1}{T}\sum\limits_{t=0}^{T-1}\mathbf{E}\|\boldsymbol{F}_t-(\boldsymbol{G}_t)^T\nabla f(\boldsymbol{g}_t)\|^2\leq \dfrac{\varepsilon^2}{9\cdot 16}  \ . 
\end{array}
\end{equation}

Using \eqref{Corollary:ErrorEstimateToBottom:Eq:g}, Lemma \ref{Lemma:Incremental-x} and the fact that $\mathbf{E}\|\boldsymbol{g}_0-g(x_0)\|^2\leq \dfrac{H_3}{S_g}$, the first inequality in \eqref{Eq:PickParameters-a-basedon-ErrorEquipartition} is satisfied if we have

\begin{equation}\label{Eq:PickParameters-a-basedon-ErrorEquipartition-g}
\dfrac{2M_g^2}{\alpha_g\beta_gL^2_\Phi}+\dfrac{2H_3}{\frac{32}{3}L_\Phi\Delta \alpha_g\gamma_g}+\dfrac{2\alpha_gH_3}{\beta_g}\leq \dfrac{1}{9\cdot 16 M_g^2L_f^2} \ .
\end{equation}

To match inequality \eqref{Eq:PickParameters-a-basedon-ErrorEquipartition-g}, we set the three term equal, so that $\dfrac{2M_g^2}{\alpha_g\beta_gL_{\Phi}^2}=\dfrac{2H_3}{\frac{32}{3}L_{\Phi}\Delta\alpha_g\gamma_g}=\dfrac{2\alpha_gH_3}{\beta_g}$, this gives \begin{equation}\label{Eq:PickParameters-alpha-g}
\alpha_g= \dfrac{M_g}{L_\Phi\sqrt{H_3}} 
\end{equation} 
and $\gamma_g=\dfrac{\beta_gH_3L_\Phi}{\frac{32}{3}\Delta M_g^2}$. The constraint \eqref{Eq:PickParameters-a-basedon-ErrorEquipartition-g} then becomes a simple constraint $\dfrac{6M_g\sqrt{H_3}}{\beta_gL_\Phi}\leq\dfrac{1}{9\cdot 16 M_g^2L_f^2}$, giving $\beta_g\geq \dfrac{864M_g^3L_f^2\sqrt{H_3}}{L_\Phi}$. So we can choose \begin{equation}\label{Eq:PickParameters-beta-gamma-g}
\beta_g= \dfrac{864M_g^3L_f^2\sqrt{H_3}}{L_\Phi} \text{ and } \gamma_g=\dfrac{81M_gL_f^2H_3^{3/2
}}{\Delta} \ .    
\end{equation}

We pick $\alpha_{\partial g}$, $\beta_{\partial g}$, $\gamma_{\partial g}$ in an exactly the same way using \eqref{Corollary:ErrorEstimateToBottom:Eq:partialg}, Lemma \ref{Lemma:Incremental-x} and the fact that $\mathbf{E}\|\boldsymbol{G}_0-\partial g(x_0)\|^2\leq \dfrac{H_2}{S_{\partial g}}$. Thus the second inequality in \eqref{Eq:PickParameters-a-basedon-ErrorEquipartition} is satisfied if we have 

\begin{equation}\label{Eq:PickParameters-a-basedon-ErrorEquipartition-partial-g}
\dfrac{2L_g^2}{\alpha_{\partial g}\beta_{\partial g}L_{\Phi}^2}+\dfrac{2H_2}{\frac{32}{3}L_{\Phi}\Delta\alpha_{\partial g}\gamma_{\partial g}}+\dfrac{2\alpha_{\partial g}H_2}{\beta_{\partial g}}\leq \dfrac{1}{9\cdot 16 M_f^2} \ .    
\end{equation}

Again we set $\dfrac{2L_g^2}{\alpha_{\partial g}\beta_{\partial g}L_{\Phi}^2}=\dfrac{2H_2}{\frac{32}{3}L_{\Phi}\Delta\alpha_{\partial g}\gamma_{\partial g}}=\dfrac{2\alpha_{\partial g}H_2}{\beta_{\partial g}}$, and this gives 

\begin{equation}\label{Eq:PickParameters-alpha-partial-g}
\alpha_{\partial g}= \dfrac{L_g}{L_\Phi\sqrt{H_2}}  
\end{equation} 
and $\gamma_{\partial g}=\dfrac{\beta_{\partial g}H_2L_{\Phi}}{\frac{32}{3}\Delta L_g^2}$. The constraint \eqref{Eq:PickParameters-a-basedon-ErrorEquipartition-partial-g} then becomes a simple constraint $\dfrac{6L_g\sqrt{H_2}}{\beta_{\partial g}L_\Phi}\leq \dfrac{1}{9\cdot 16 M_f^2}$, giving $\beta_{\partial g}\geq\dfrac{864M_f^2L_g\sqrt{H_2}}{L_\Phi}$. So we can choose 

\begin{equation}\label{Eq:PickParameters-beta-gamma-partial-g}
\beta_{\partial g}= \dfrac{864M_f^2L_g\sqrt{H_2}}{L_\Phi} \text{ and } \gamma_{\partial g}=\dfrac{81M_f^2H_2^{3/2
}}{\Delta L_g} \ .    
\end{equation}

Finally we can pick the parameters for the $\boldsymbol{F}$-iteration. 

Indeed we first observe that we can write \eqref{Lemma:Increment-boldg:Eq:Estimate} as 

\begin{equation}\label{Eq:Increment-boldg-Via-g-Estimate}
\begin{array}{ll}
\dfrac{1}{T}\sum\limits_{t=0}^{T-1}\mathbf{E}\|\boldsymbol{g}_{t+1}-\boldsymbol{g}_t\|^2&\stackrel{(a)}{\leq} \dfrac{6}{T}\sum\limits_{t=0}^{T-1}\mathbf{E}\|\boldsymbol{g}_t-g(x_t)\|^2+3M_g^2\left[\dfrac{2}{B_g}+1\right]\dfrac{1}{T}\sum\limits_{t=0}^{T-1}\mathbf{E}\|x_{t+1}-x_t\|^2+\dfrac{6a_g^2H_3}{B_g} 
\\
&\stackrel{(b)}{\leq} \dfrac{6\varepsilon^2}{9\cdot 16 M_g^2L_f^2}+3M_g^2\left(1+\dfrac{2L_\Phi}{864M_g^3L_f^2\sqrt{H_3}}\varepsilon\right)\dfrac{1}{L_{\Phi}^2}\varepsilon^2+\dfrac{6}{864L_{\Phi}M_gL_f^2\sqrt{H_3}}\varepsilon^3
\\
&=\left(\dfrac{1}{24M_g^2L_f^2}+\dfrac{3M_g^2}{L_{\Phi}^2}\right)\varepsilon^2+\dfrac{1}{72L_{\Phi}M_gL_f^2\sqrt{H_3}}\varepsilon^3 \ .
\end{array}
\end{equation}
Here in (a) we used \eqref{Lemma:Increment-boldg:Eq:Step1} and in (b) we used \eqref{Eq:PickParameters-alpha-g}, \eqref{Eq:PickParameters-beta-gamma-g}. 

Similarly we have 

\begin{equation}\label{Eq:Increment-boldG-Via-G-Estimate}
\begin{array}{ll}
\dfrac{1}{T}\sum\limits_{t=0}^{T-1}\mathbf{E}\|\boldsymbol{G}_{t+1}-\boldsymbol{G}_t\|_F^2
&\stackrel{(a)}{\leq}
\dfrac{6}{T}\sum\limits_{t=0}^{T-1}\mathbf{E}\|\boldsymbol{G}_t-\partial g(x_t)\|_F^2+3L_g^2\left[\dfrac{2}{B_{\partial g}}+1\right]\dfrac{1}{T}\sum\limits_{t=0}^{T-1}\mathbf{E}\|x_{t+1}-x_t\|^2+\dfrac{6(a_{\partial g})^2H_2}{B_{\partial g}}
\\
& \stackrel{(b)}{\leq}\dfrac{6\varepsilon^2}{9\cdot 16M_f^2}+3L_g^2\left(1+\dfrac{2L_\Phi}{864M_f^2L_g\sqrt{H_2}}\varepsilon\right)\dfrac{1}{L_{\Phi}^2}\varepsilon^2+\dfrac{6L_g\varepsilon^3}{864L_{\Phi}M_f^2\sqrt{H_2}}
\\
&=\left(\dfrac{1}{24M_f^2}+\dfrac{3L_g^2}{L_{\Phi}^2}\right)\varepsilon^2+\dfrac{L_g}{72L_{\Phi}M_f^2\sqrt{H_2}}\varepsilon^3 \ .
\end{array}
\end{equation}
Here in (a) we used \eqref{Lemma:Increment-boldG:Eq:Step1} and in (b) we used \eqref{Eq:PickParameters-alpha-partial-g},  \eqref{Eq:PickParameters-beta-gamma-partial-g}.

We pick the desired precision parameter $\varepsilon$ such that \begin{equation}
0<\varepsilon<\min\left(1, 72L_{\Phi}M_gL_f^2\sqrt{H_3}, \dfrac{72L_{\Phi}M_f^2\sqrt{H_2}}{L_g}\right) \ ,
\end{equation}
so that 

\begin{equation}\label{Eq:Increment-g-boldg-Estimate}
\begin{array}{ll}
\dfrac{1}{T}\sum\limits_{t=0}^{T-1}\mathbf{E}\|\boldsymbol{g}_{t+1}-\boldsymbol{g}_t\|^2&\leq \left(\dfrac{1}{24M_g^2L_f^2}+\dfrac{3M_g^2}{L_{\Phi}^2}+1\right)\varepsilon^2 \ ,
\\
\dfrac{1}{T}\sum\limits_{t=0}^{T-1}\mathbf{E}\|\boldsymbol{G}_{t+1}-\boldsymbol{G}_t\|_F^2 &\leq \left(\dfrac{1}{24M_f^2}+\dfrac{3L_g^2}{L_{\Phi}^2}+1\right)\varepsilon^2 \ .
\end{array}
\end{equation}
We put these two estimates in \eqref{Lemma:ErrorEstimateToBottom:gradPhi:Eq:Error:Step1}, and notice that $2M_g+\dfrac{L_g\eta\varepsilon}{a_{\partial g}}=2M_g+\dfrac{L_g}{L_{\Phi}\frac{L_g}{L_{\Phi}}\sqrt{H_2}}=2M_g+\sqrt{H_2}$ so that the coefficient of $\dfrac{1}{a_\Phi B_f}$ is  estimated by 
\begin{equation}\label{Eq:boldF-Estimate-Coefficient-1-Over-aB}
\begin{array}{ll}
&\dfrac{4}{B_f}\left[M_f^2\sum\limits_{t=0}^{T-1} \mathbf{E}\|\boldsymbol{G}_{t+1}-\boldsymbol{G}_t\|_F^2+L_f^2\left(2M_g+\dfrac{L_g\eta\varepsilon}{a_{\partial g}}\right)^2\sum\limits_{t=0}^{T-1}\mathbf{E}\|\boldsymbol{g}_{t+1}-\boldsymbol{g}_t\|^2\right]    
\\
\leq &\left[\dfrac{4M_f^2}{B_f}\left(\dfrac{1}{24M_f^2}+\dfrac{3L_g^2}{L_{\Phi}^2}+1\right)+\dfrac{4L_f^2}{B_f}(2M_g+\sqrt{H_2})\left(\dfrac{1}{24M_g^2L_f^2}+\dfrac{3M_g^2}{L_{\Phi}^2}+1\right)\right]T\varepsilon^2 \ .
\end{array}
\end{equation}

Denote
\begin{equation}\label{Eq:K-0}
K_0=\dfrac{4M_f^2}{B_f}\left(\dfrac{1}{24M_f^2}+\dfrac{3L_g^2}{L_{\Phi}^2}+1\right)+\dfrac{4L_f^2}{B_f}(2M_g+\sqrt{H_2})\left(\dfrac{1}{24M_g^2L_f^2}+\dfrac{3M_g^2}{L_{\Phi}^2}+1\right) \ ,
\end{equation}
then \eqref{Lemma:ErrorEstimateToBottom:gradPhi:Eq:Error:Step1} gives 
$$\dfrac{1}{T}\sum\limits_{t=0}^{T-1}\mathbf{E}\|\boldsymbol{F}_t-(\boldsymbol{G}_t)^T\nabla f(\boldsymbol{g}_t)\|^2\leq \dfrac{K_0}{a_{\Phi}B_f}\varepsilon^2+\dfrac{2a_{\Phi}(2M_g+\sqrt{H_2})H_1}{B_f}+\dfrac{M_g^2H_1}{Ta_\Phi S_f} \ .$$
Here we used $\mathbf{E}\|\boldsymbol{F}_0-(\boldsymbol{G}_0)\nabla f(\boldsymbol{g}_0)\|^2\leq\dfrac{M_g^2H_1}{S_f}$. 
Taking into account that $T=\dfrac{32}{3}L_{\Phi}\Delta\varepsilon^{-2}$, $a_\Phi=\alpha_{\Phi}\varepsilon$ and $B_{f}=\beta_{f}\varepsilon^{-1}$, $S_{f}=\gamma_{f}\varepsilon^{-1}$, we see from here that the third inequality in \eqref{Eq:PickParameters-a-basedon-ErrorEquipartition} is reduced to 

\begin{equation}\label{Eq:PickParameters-a-basedon-ErrorEquipartition-F}
\dfrac{K_0}{\alpha_{\Phi}\beta_f}+\dfrac{M_g^2H_1}{\frac{32}{3}L_{\Phi}\Delta\alpha_{\Phi}\gamma_f}+\dfrac{2\alpha_{\Phi}(2M_g+\sqrt{H_2})H_1}{\beta_f}\leq \dfrac{1}{9\cdot 16} \ .
\end{equation}
Again we set $\dfrac{K_0}{\alpha_{\Phi}\beta_f}=\dfrac{M_g^2H_1}{\frac{32}{3}L_{\Phi}\Delta\alpha_{\Phi}\gamma_f}=\dfrac{2\alpha_{\Phi}(2M_g+\sqrt{H_2})H_1}{\beta_f}$, and this gives

\begin{equation}\label{Eq:PickParameters-alpha-Phi}
\alpha_{\Phi}= \sqrt{\dfrac{K_0}{2(2M_g+\sqrt{H_2})H_1}}
\end{equation} 
and $\gamma_{f}=\dfrac{M_g^2H_1}{\frac{32}{3}L_{\Phi}\Delta K_0}\beta_f$. The constraint \eqref{Eq:PickParameters-a-basedon-ErrorEquipartition-F} then becomes a simple constraint $\dfrac{3K_0}{\sqrt{\frac{K_0}{2(2M_g+\sqrt{H_2})H_1}}\beta_f}\leq \dfrac{1}{9\cdot 16}$, giving $\beta_{f}\geq 432\sqrt{2(2M_g+\sqrt{H_2})H_1K_0}$. So we can choose 

\begin{equation}\label{Eq:PickParameters-beta-gamma-F}
\beta_{f}= 432\sqrt{2(2M_g+\sqrt{H_2})H_1K_0} \text{ and } \gamma_{f}=\dfrac{M_g^2H_1}{\frac{32}{3}L_{\Phi}\Delta K_0}432\sqrt{2(2M_g+\sqrt{H_2})H_1K_0} \ .    
\end{equation}

With 
\eqref{Eq:PickParameters-beta-gamma-g}, \eqref{Eq:PickParameters-beta-gamma-partial-g}, \eqref{Eq:PickParameters-beta-gamma-F}, the total IFO complexity is given by 

\begin{equation}\label{Eq:IFO-Complexity-PreciseParameters}
\begin{array}{ll}
\text{IFO}&=
(\gamma_g+\gamma_{\partial g}+\gamma_{f})\varepsilon^{-1}+\dfrac{32}{3}L_{\Phi}\Delta(\beta_g+\beta_{\partial g}+\beta_f)\varepsilon^{-3}
\\
&=\left(\dfrac{81M_gL_f^2H_3^{3/2
}}{\Delta}+\dfrac{81M_f^2H_2^{3/2
}}{\Delta L_g}+\dfrac{M_g^2H_1}{\frac{32}{3}L_{\Phi}\Delta K_0}432\sqrt{2(2M_g+\sqrt{H_2})H_1K_0}\right)\cdot \varepsilon^{-1}
\\
& \qquad +\left(\dfrac{864M_g^3L_f^2\sqrt{H_3}}{L_\Phi}+\dfrac{864M_f^2L_g\sqrt{H_2}}{L_\Phi}+432\sqrt{2(2M_g+\sqrt{H_2})H_1K_0}\right)\cdot \varepsilon^{-3} \ .
\end{array}
\end{equation}
\end{proof}

\section{Proof of Corollary 3.1}

\begin{proof}

We claim that if we choose $\eta\sim \mathcal{O}(1)$, $a_{g}, a_{\partial g}, a_{\Phi}\sim \mathcal{O}(\varepsilon)$, $B_g, B_{\partial g}, B_f, S_g, S_{\partial g}, S_f\sim \mathcal{O}(\varepsilon^{-1})$ and $T\sim \mathcal{O}(\varepsilon^{-2})$, then we have 
\begin{equation}\label{Lemma:AsymptoticChoice-a-batchsize:Eq:Asymptotic-eps-estimate:g}
\dfrac{1}{T}\sum\limits_{t=0}^{T-1}\mathbf{E}\|\boldsymbol{g}_t-g(x_t)\|^2 \lesssim \mathcal{O}(\varepsilon^2) \ ,
\end{equation}

\begin{equation}\label{Lemma:AsymptoticChoice-a-batchsize:Eq:Asymptotic-eps-estimate:partialg}
\dfrac{1}{T}\sum\limits_{t=0}^{T-1}\mathbf{E}\|\boldsymbol{G}_t-\partial g(x_t)\|_F^2 \lesssim \mathcal{O}(\varepsilon^2) \ ,
\end{equation}

\begin{equation}\label{Lemma:AsymptoticChoice-a-batchsize:Eq:Asymptotic-eps-estimate:g-partialg-increment}
\dfrac{1}{T}\sum\limits_{t=0}^{T-1}\mathbf{E}\|\boldsymbol{g}_{t+1}-\boldsymbol{g}_t\|^2,
\dfrac{1}{T}\sum\limits_{t=0}^{T-1}\mathbf{E}\|\boldsymbol{G}_{t+1}-\boldsymbol{G}_t\|_F^2 \lesssim \mathcal{O}(\varepsilon^2) \ ,
\end{equation}

\begin{equation}\label{Lemma:AsymptoticChoice-a-batchsize:Eq:Asymptotic-eps-estimate:gradPhi}
\dfrac{1}{T}\sum\limits_{t=0}^{T-1}\mathbf{E}\|\boldsymbol{F}_t-\boldsymbol{G}_t\nabla f(\boldsymbol{g}_t)\|^2
\lesssim \mathcal{O}(\varepsilon^2) \ .
\end{equation}

Take \eqref{Lemma:AsymptoticChoice-a-batchsize:Eq:Asymptotic-eps-estimate:g} as an example, we recall that by \eqref{Corollary:ErrorEstimateToBottom:Eq:g} we have
$$\dfrac{1}{T}\sum\limits_{t=0}^{T-1}\mathbf{E}\|\boldsymbol{g}_t-g(x_t)\|^2
\leq\dfrac{2M_g^2}{Ta_gB_g}\sum\limits_{t=0}^{T-1}\mathbf{E}\|x_{t+1}-x_t\|^2+\dfrac{a_gH_3}{B_g}+\dfrac{1}{Ta_g}\mathbf{E}\|\boldsymbol{g}_0-g(x_0)\|^2 \ ,$$

In the above, when $a_g\sim \mathcal{O}(\varepsilon)$ and $B_g\sim \mathcal{O}(\varepsilon^{-1})$, we see that $a_gB_g\sim \mathcal{O}(1)$. Since $\|x_{t+1}-x_t\|^2\lesssim \mathcal{O}(\varepsilon^2)$ by our choice of $\eta$ and Lemma \ref{Lemma:Incremental-x}, we have $\dfrac{1}{T}\dfrac{2M_g^2}{a_gB_g}\sum\limits_{t=0}^{T-1}\mathbf{E}\|x_{t+1}-x_t\|^2\lesssim \mathcal{O}(\varepsilon^2)$, which settles the first term. For the second term, since $\dfrac{a_g}{B_g}\sim \mathcal{O}(\varepsilon^2)$, it is also settled. For the last term, notice that $\mathbf{E}\|\boldsymbol{g}_0-g(x_0)\|^2=\dfrac{1}{S_g}\mathbf{E}\|g_i(x_0)-g(x_0)\|^2\leq \dfrac{H_3}{S_g}$ by Assumption \ref{Assumption:FiniteVariance} and our with replacement sampling of $\mathcal{S}_0^g$. Then we have $\dfrac{1}{Ta_g}\mathbf{E}\|\boldsymbol{g}_0-g(x_0)\|^2\leq \dfrac{1}{T}\dfrac{H_3}{a_gS_g}\lesssim \mathcal{O}(\dfrac{1}{T})=\mathcal{O}(\varepsilon^2)$, which is again settled. 
We then see that \eqref{Corollary:ErrorEstimateToBottom:Eq:partialg} goes in exactly the same way. 

It is easy to see that \eqref{Lemma:AsymptoticChoice-a-batchsize:Eq:Asymptotic-eps-estimate:g-partialg-increment} is also valid just by \eqref{Lemma:Increment-boldg:Eq:Estimate} and \eqref{Lemma:Increment-boldG:Eq:Estimate}, as well as the arguments we used above.

Finally for \eqref{Lemma:AsymptoticChoice-a-batchsize:Eq:Asymptotic-eps-estimate:gradPhi}, we have to notice that by our choice of the asymptotics, we have $\dfrac{L_g\eta\varepsilon}{a_{\partial g}}\sim \mathcal{O}(1)$, as well as the fact that $\mathbf{E}\|\boldsymbol{F}_0-(\boldsymbol{G}_0)^T\nabla f(\boldsymbol{g}_0)\|^2\leq \dfrac{M_g^2H_1}{S_{f}}$, these combined with \eqref{Lemma:AsymptoticChoice-a-batchsize:Eq:Asymptotic-eps-estimate:g-partialg-increment} enables the validity of \eqref{Lemma:AsymptoticChoice-a-batchsize:Eq:Asymptotic-eps-estimate:gradPhi}.

Thus by Proposition 3.1 we see that the IFO complexity of Algorithm 1 to reach $\cO(\ve)$-accuracy is of order $\lesssim \mathcal{O}(\varepsilon^{-3})$.

\end{proof}

\section{More Experiments}

\subsection{Value Function Evaluation in Reinforcement Learning}

We carry another experiment for STORM-Compositional on the problem of value function evaluation in reinforcement learning, same as Section 4.2 in \cite{arXivSARAH-SCGD}. The target is to find the value function $V^\pi(s)$ of state $s$ under policy $\pi$ for an underlying Markov Decision Process. The value function $V^\pi(s)$ can be evaluated through Bellman equation (see \cite{[Sutton-Barto]})

\[V^\pi(s_1)=\E\left[r_{s_1,s_2}+\gm V^\pi(s_2)|s_1\right] \ ,\]
for all $s_1, s_2, ..., s_n\in \cS$, where $\cS$ represents the set of available states and $|\cS|=n$, and $r_{s_1,s_2}$ is the reward function. The value function evaluation task can be formulated as a minimization problem of the square loss

\[\sum\li_{s\in \cS}\left(V^\pi(s)-\sum\li_{s'\in \cS}P_{s,s'}\left(r_{s,s'}+\gm V^\pi(s')\right)\right)^2 \ .\]
 
Here $P_{s,s'}$ is the transition probability. Set $\widehat{V}^\pi(s)=P_{s,s'}\left(r_{s,s'}+\gm V^\pi(s')\right)$, then the above problem can be formulated as compositional optimization problem with the choice of $g$ and $f$ as following
(see \cite[Section 4.2]{arXivSARAH-SCGD})

\[g(s)=\left[V^\pi(s_1), ..., V^\pi(s_n), \widehat{V}^\pi(s_1), ..., \widehat{V}^\pi(s_n)\right] \ ,\]

\[f(w)=\sum\li_{i=1}^n (w_i-w_{n+i})^2 \ ,\]
where $w\in \R^n$ is the vector with the elements in $g(s)$ as components.

We use the same reinforcement learning model as in \cite[Section 4.2]{arXivSARAH-SCGD}, which has $400$ states and $10$ actions for each state, and we use the exact same way of sampling the transition probability as well as the reward function. In our experiment for STORM-Compositional, we use the same parameter settings as \cite{arXivSARAH-SCGD} for SARAH-C, VRSC-PG, SCGD and ASC-PG. But we do a simple parameter setting for our STORM-Compositional just using the orders of parameters, and also taking into account that we have to match the approximately same amount of IFO's for SARAH-C. We take $\eta=0.1, \ve=0.1, S_g=S_{\pt g}=100, S_{f}=1, B_g=B_{\pt g}=20, B_{f}=1, a_g=a_{\pt g}=a_{\Phi}=0.1$ and the results are plotted in Figure \ref{Fig_RL}, where left column are for $\Phi(x)-\Phi^*$ and right column are for the gradient norms as functions of IFO queries. It is seen that even for such simple and straightforward parameter setting, STORM-Compositional behaves much better than SARAH-C and other compositional optimization algorithms after sufficient numbers of iterations. 

\begin{figure}
\captionsetup{margin=0cm, justification=centering}
\centering
\includegraphics[height=8cm,  width=\textwidth]{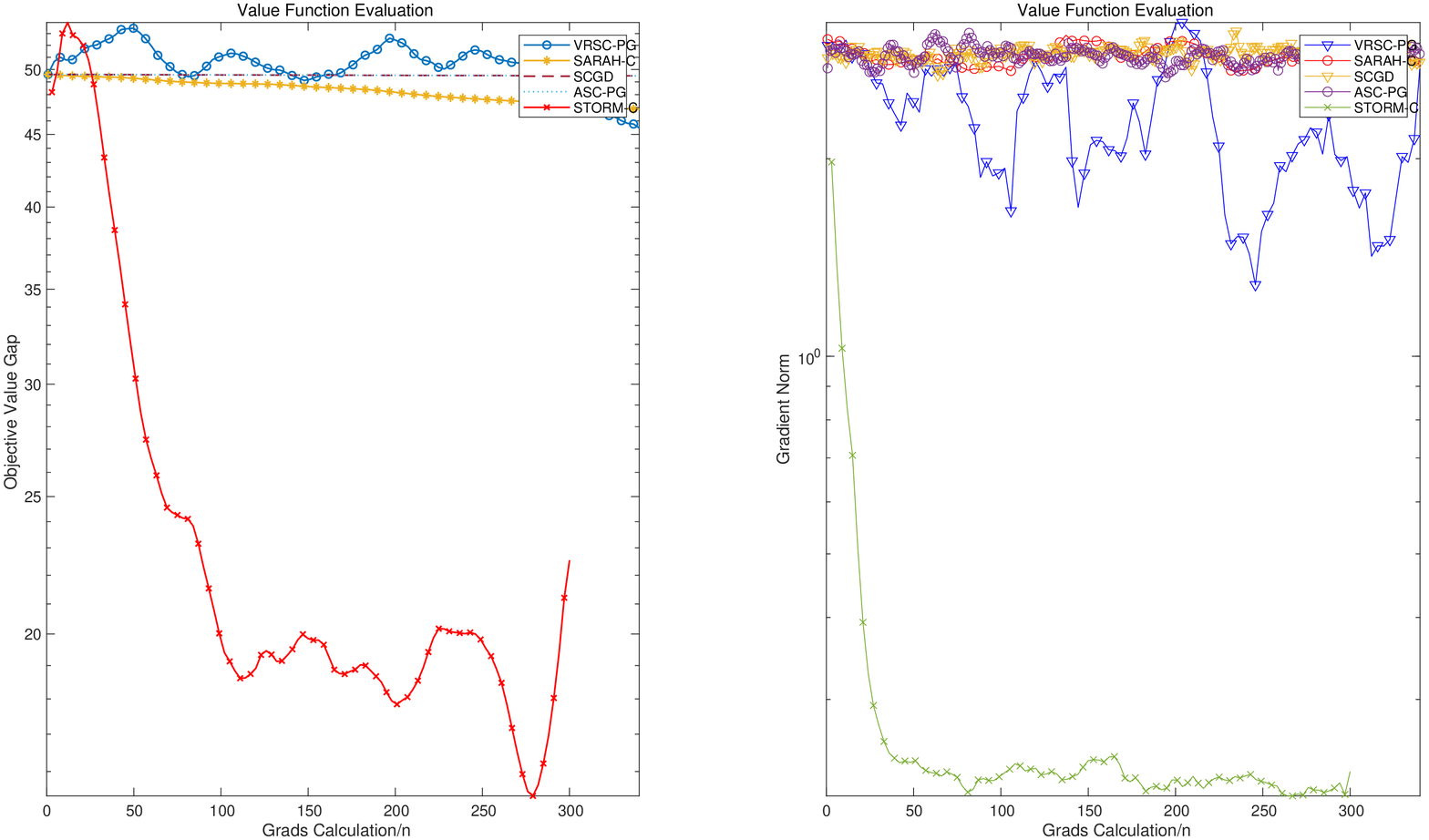}
\caption{STORM-Compositional compared with other compositional optimization algorithms for value function 
evaluation problem in Reinforcement Learning: Left Column: Objective Function Value Gap (vertical axis) vs. Gradient Calculations (horizontal axis); Right Column: Objective Function Gradient Norm (vertical axis) vs. Gradient Calculation (horizontal axis).}
\label{Fig_RL}
\end{figure}

\subsection{Stochastic Neighbor Embedding}

The Stochastic Neighbor Embedding (SNE, see \cite{HintonSNE}) is a dimension reduction method used very often in image classification tasks. The problem here can be formulated as follows: Let $z_1,...,z_n\in \R^N$ be a family of data points in a very high dimensional ($N>\!\!>1$) space. We want to find their low-dimensional embeddings $x_1,...,x_n\in \R^d$ so that the affinity $d(z_i, z_j)$ is preserved as much as we can. Here the affinity map can be taken as say $d(z_i, z_j)=\exp(-\|z_i-z_j\|^2/2\sm_i^2)$ for a sequence of standard deviations $\sm_1,...,\sm_n>0$. We choose the affinity map between the $y$'s as $\exp(-\|x_i-x_j\|^2/2)$. Then we can define the similarity between the $z$'s via the conditional probability

\begin{equation}\label{Eq:Similarity-x}
p_{j|i}=\dfrac{\exp(-\|z_i-z_j\|^2/2\sm_i^2)}{\sum\li_{k\neq i} \exp(-\|z_i-z_k\|^2/2\sm_i^2)} \ .
\end{equation}
Similarly, we can deifine the similarity between the $x$'s via the conditional probability

\begin{equation}\label{Eq:Similarity-y}
q_{j|i}=\dfrac{\exp(-\|x_i-x_j\|^2/2)}{\sum\li_{k\neq i}\exp(-\|x_i-x_k\|^2/2)} \ .
\end{equation}

Thus the problem can be formulated as minimizing the KL-divergence between the $p$ and $q$'s, i.e.

\begin{equation}\label{Eq:Objective-KL}
\min\li_{x_1,...,x_n\in \R^d}KL(x_1,...,x_n)=\sum\li_{i=1}^n \sum\li_{j=1, j\neq i}^n p_{j|i}\log\left(\dfrac{p_{j|i}}{q_{j|i}}\right) \ .
\end{equation}

In \cite{SCVR} the authors have reduced the above objective function as a compositional optimization prblem with the objective function

\begin{equation}\label{Eq:Objective-Compositional}
\begin{array}{ll}
\Phi(x_1,...,x_n)& =\dfrac{1}{n}\sum\li_{j=1}^n
f_j\left(\dfrac{1}{n}\sum\li_{i=1}^n g_i(x_1,...,x_n)\right)
\\
&=\sum\li_{j=1}^n
\left(\sum\li_{i=1}^n p_{j|i}\left(\|x_i-x_j\|^2+\log\left(\sum\li_{k=1}^n e^{-\|x_i-x_k\|^2}-1\right)\right)\right) \ .
\end{array}
\end{equation} 

Here we have, for $i=1,2,...,n$ and $j=1,2,...,n$, that 

\begin{equation}\label{Eq:Compositional-f}
\begin{array}{ll}
g_i(x)&=\left[x_1,...,x_n, ne^{-\|x_1-x_i\|^2}-1,...,
ne^{-\|x_n-x_i\|^2}-1\right]^\top\in \R^{dn+n} \ ,
\\
f_j(y)&=n\sum\li_{i=1}^n p_{j|i}\left(\|y_i-y_j\|^2+\log(y_{n+i})\right) \ .
\end{array}
\end{equation} 
Here we treat $x=(x_1,...,x_n)\in \R^{d\times n}$ as a matrix with column vectors $x_1,...,x_n\in \R^d$ and $y=(y_1,...y_n, y_{n+1},...,y_{2n})^\top\in \R^{dn+n}$ is such that

\begin{equation}\label{Eq:y}
y_k=\left\{\begin{array}{ll}
x_k & \text{ when } k=1,2,...,n \ ;
\\
ne^{-\|x_k-x_i\|^2}-1 &\text{ when } k=n+1,...,2n \ .
\end{array}\right.
\end{equation}
We can also treat $y$ as a matrix in $\R^{(d+1)\times n}$ such that $y=\begin{pmatrix}y_1,...,y_n\\
y_{n+1},...,y_{2n}
\end{pmatrix}$. Notice that here each $y_1,...,y_n$ is a column vector of dimension $d$ and the last row is a row vector $(y^{(2)})^\top=(y_{n+1},...,y_{2n})$. 

If we treat $x$ as a vector in $\R^{dn}$, then the Jacobian $\pt g_i(x)$ is a matrix of the form

\begin{equation}\label{Eq:partial-g}
\pt g_i(x)=
\begin{pmatrix}
I_d & 0 & 0 & ... & 0
\\
0 & I_d & 0 & ... & 0
\\
0 & 0 & I_d & ... & 0
\\
... & ... & ... & ... & ... 
\\
0 & 0 & 0   & ... & I_d
\\
\grad_{x_1}^\top g_{i,n+1}(x) & 
\grad_{x_2}^\top g_{i,n+1}(x) &
\grad_{x_3}^\top g_{i,n+1}(x) & ... & \grad_{x_n}^\top g_{i,n+1}(x)
\\
... & ... & ... & ... & ... 
\\
\grad_{x_1}^\top g_{i,n+n}(x) & 
\grad_{x_2}^\top g_{i,n+n}(x) &
\grad_{x_3}^\top g_{i,n+n}(x) & ... & \grad_{x_n}^\top g_{i,n+n}(x)
\end{pmatrix}_{(dn+n)\times dn} \ .
\end{equation}
Here we denote $g_{i, n+k}(x)=ne^{-\|x_k-x_i\|^2}-1$ and $\grad_{x_l}g_{i, n+k}(x)$ is the gradient vector with respect to $x_l$, which is in $\R^d$. It is easy to calculate that we have 

\begin{equation}\label{Eq:grad-vector-g-nto2n}
\begin{array}{l}
\grad_{x_k} g_{i, n+k} (x)=-2ne^{-\|x_k-x_i\|^2}(x_k-x_i) \ , 
\\
\grad_{x_i} g_{i, n+k}(x)=2ne^{-\|x_k-x_i\|^2}(x_k-x_i) \ ,
\\
\grad_{x_j} g_{i, n+k}(x)=0 \ , \text{ for } j \neq k, i \ .
\end{array}
\end{equation}

We then calculate $\grad f_j(y)$ as
$\grad f_j(y)=((\grad f_j^{(1)}(y))^\top, (\grad f_j^{(2)}(y))^\top)^\top$, so that  

\begin{equation}\label{Eq:grad-f-part1}
\begin{array}{ll}
\grad f_j^{(1)}(y)&
=\left(2np_{j|1}(y_1-y_j)^\top, ..., 
2np_{j|j-1}(y_{j-1}-y_j)^\top, \dfrac{\pt f_j^{(1)}}{\pt y_j}(y), \right.
\\
 & \qquad \qquad \left.
2np_{j|j+1}(y_{j+1}-y_j)^\top, ...
,
2np_{j|n}(y_n-y_j)^\top\right)^\top \ ,
\end{array}
\end{equation}
where 
\begin{equation}\label{Eq:grad-f-part1-j}
\dfrac{\pt f_j^{(1)}}{\pt y_j}(y)=-2n\sum\li_{k=1}^n p_{j|k}(y_k-y_j) \ .
\end{equation}
Moreover, 
\begin{equation}\label{Eq:grad-f-part2}
\grad f_j^{(2)}(y)
=\left(\dfrac{np_{j|1}}{y_{n+1}},...,\dfrac{np_{j|n}}{y_{n+n}}\right)^\top \ .
\end{equation}
It is easy to see that $\grad f_j^{(1)}(y)\in \R^{dn}$ and $\grad f_j^{(2)}(y)\in \R^n$. Thus $\grad f_j(y)\in \R^{dn+n}$. 

Based on these, we calculate $(\pt g_i(x))^T\grad f_j(y)$ as 

\begin{equation}\label{Eq:gradPhi}
\begin{array}{ll}
&(\pt g_i(x))^T\grad f_j(y)
\\
= &\grad f_j^{(1)}(y)+\begin{pmatrix}
\grad_{x_1}g_{i, n+1}(x)&...&\grad_{x_1}g_{i, n+n}(x)
\\
...&...&...
\\
\grad_{x_n}g_{i, n+1}(x)&...&\grad_{x_n}g_{i, n+n}(x)
\end{pmatrix}\grad f_j^{(2)}(y) 
\\
=&\grad f_j^{(1)}(y)+\begin{pmatrix}
\begin{pmatrix}\grad_{x_1}g_{i, n+1}(x)&...&\grad_{x_1}g_{i, n+n}(x)
\end{pmatrix}\grad f_j^{(2)}(y) 
\\
...
\\
\begin{pmatrix}\grad_{x_n}g_{i, n+1}(x)&...&\grad_{x_n}g_{i, n+n}(x)
\end{pmatrix}\grad f_j^{(2)}(y) 
\end{pmatrix}\ .
\end{array}
\end{equation}

We can treat $G_i(x)=\begin{pmatrix}
\grad_{x_1}g_{i, n+1}(x)&...&\grad_{x_1}g_{i, n+n}(x)
\\
...&...&...
\\
\grad_{x_n}g_{i, n+1}(x)&...&\grad_{x_n}g_{i, n+n}(x)
\end{pmatrix}$ as a tensor consisting of matrices $G_i(x)=\begin{pmatrix}G_{i,1}(x)\\...\\G_{i,n}(x)\end{pmatrix}$ such that $G_{i,l}(x)=\begin{pmatrix}
\grad_{x_l}g_{i, n+1}(x)&...&\grad_{x_l}g_{i, n+n}(x)
\end{pmatrix}$
for $l=1,2,...,n$. Taking into account \eqref{Eq:grad-vector-g-nto2n} we know that each component $G_{i,l}(x)$ of this tensor is given by a matrix consisting of $n$ column vectors in $\R^d$ such that the $i$-th column vector is $2ne^{-\|x_l-x_i\|^2}(x_l-x_i)$ and the $l$-th column vector is $-2ne^{-\|x_l-x_i\|^2}(x_l-x_i)$ and all elsewhere are $0$.

Following this perspective, we can also treat $\grad f_j^{(1)}(y)$ as a tensor equivalent to a matrix with $n$ components, so that $\grad f_j^{(1)}(y)=\begin{pmatrix}
\grad f_{j,1}^{(1)}(y) 
\\
...
\\
\grad f_{j,n}^{(1)}(y) 
\end{pmatrix}$. Here we have $\grad f_{j, l}^{(1)}(y)=2np_{j|l}(y_l-y_j)$ for $l=\{1,2,...,n\}\backslash\{j\}$ and $\grad f_{j, j}^{(1)}(y)=-2n\sum\li_{k=1}^n p_{j|k}(y_k-y_j)$. Notice that since $y_i=x_i$ for $i=1,2,...,n$, we indeed have 

$$\grad f_{j, l}^{(1)}(y)=2np_{j|l}(x_l-x_j) \text{ for } l=\{1,2,...,n\}\backslash\{j\} \ , \grad f_{j, j}^{(1)}(y)=-2n\sum\li_{k=1}^n p_{j|k}(x_k-x_j) \ .$$

Based on these above calculations, we have performed the experiment on SNE for our STORM-Compositional and compared it with VRSC-PG, SARAH-C, SCGD and ASC-PG. For the experiments on VRSC-PG, SARAH-C, SCGD and ASC-PG, we use the same set-up as in \cite{arXivSARAH-SCGD}. For STORM-Compositional, without much tuning of the parameters, we take $\eta=0.1, \ve=0.1, S_g=S_{\pt g}=S_{f}=100, B_g=B_{\pt g}=B_{f}=100, a_g=a_{\pt g}=a_{\Phi}=0.01$ and the results are plotted in Figure \ref{Fig_SNE}, where left column are for $\Phi(x)-\Phi^*$ and right column are for the gradient norms as functions of IFO queries. It is seen that even for such simple and straightforward parameter setting, STORM-Compositional behaves much better than SARAH-C and other compositional optimization algorithms after sufficient numbers of iterations.

\begin{figure}
\captionsetup{margin=0cm, justification=centering}
\centering
\includegraphics[height=8cm,  width=\textwidth]{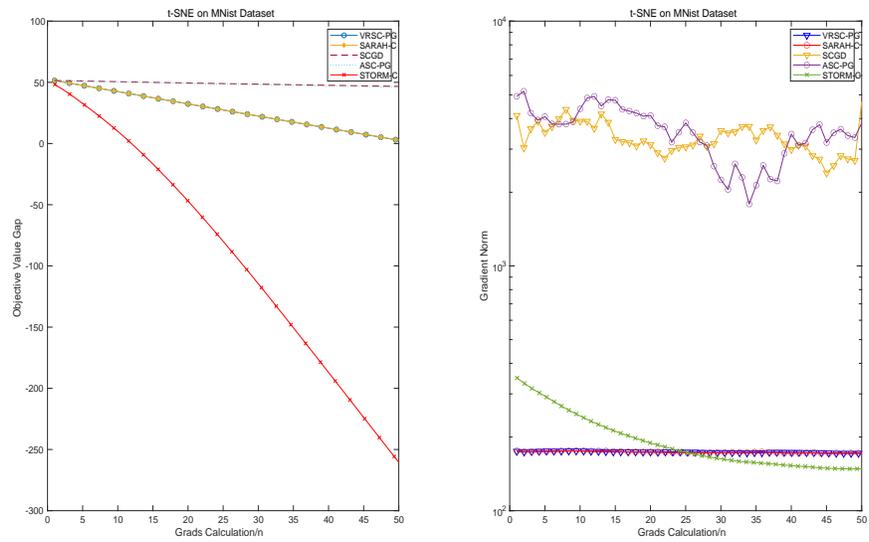}
\caption{STORM-Compositional compared with other compositional optimization algorithms for Stochastic Neighbor Embedding Problem: Left Column: Objective Function Value Gap (vertical axis) vs. Gradient Calculations (horizontal axis); Right Column: Objective Function Gradient Norm (vertical axis) vs. Gradient Calculation (horizontal axis).}
\label{Fig_SNE}
\end{figure}

\end{document}